\documentclass[a4, 12pt]{amsart}
\UseRawInputEncoding
\usepackage{mathrsfs}
\usepackage{amsmath}
\usepackage{amsfonts}
\usepackage{amssymb}
\usepackage{amsmath,amssymb,amsthm}
\usepackage{amsmath,amssymb,amsthm,amscd}
\usepackage[frame,cmtip,arrow,matrix,line,graph,curve]{xy}
\usepackage{graphpap, color}
\usepackage[mathscr]{eucal}
\usepackage{color}
\usepackage{verbatim}
\usepackage{hyperref}
\usepackage{cite}
\numberwithin{equation}{section} \numberwithin{equation}{section}
\newtheorem{thm}{Theorem}[section]
\newtheorem{lem}{Lemma}[section]

\newtheorem{corr}{Corollary}[section]
\newtheorem{rem}{Remark}[section]

\newtheorem*{ack}{Acknowledgment}

\renewcommand{\oddsidemargin}{5mm}

\makeatletter

\title[Eigenvalue Inequalities for the Clamped Plate Problem]{Eigenvalue Inequalities for the \\ Clamped Plate Problem of $\mathfrak{L}^{2}_{\nu}$ Operator}
\author[L. Zeng ]{Lingzhong Zeng }

\address{Lingzhong Zeng
\\  \newline \indent College of Mathematics and Statistics
\\  \newline \indent   Jiangxi Normal University, Nanchang 330022,  China. lingzhongzeng@yeah.net }

\begin{document}
\maketitle

\begin{abstract}   $\mathfrak{L}_{II}$ operator is introduced by Y.-L. Xin (\emph{Calculus of Variations and
Partial Differential Equations. 2015, \textbf{54}(2):1995-2016)}, which is an important  extrinsic elliptic differential operator of divergence type and has profound geometric meaning.  In this paper, we extend $\mathfrak{L}_{II}$ operator to more general  elliptic differential operator $\mathfrak{L}_{\nu}$, and investigate the clamped plate problem of bi-$\mathfrak{L}_{\nu}$ operator, which is denoted by $\mathfrak{L}_{\nu}^{2}$, on the complete Riemannian manifolds.  A general formula of eigenvalues for the $\mathfrak{L}_{\nu}^{2}$ operator is established. Applying this formula, we estimate the eigenvalues with lower order on the Riemannian manifolds. As some further applications, we establish some eigenvalue inequalities for this operator on the translating solitons with respect to the mean curvature flows, submanifolds of the Euclidean spaces, unit spheres and projective spaces. In particular,  for the case of translating solitons, all of the eigenvalue inequalities are universal.
\end{abstract}


\section{Introduction}
Suppose that $(\mathcal{M}^{n},g)$ is an $n$-dimensional, complete, noncompact Riemannian manifold with smooth metric $g$ and $\Omega$ is a bounded domain with piecewise smooth boundary $\partial\Omega$ on $\mathcal{M}^{n}$. We consider the following fixed membrane problem of Laplacian on $\mathcal{M}^{n}$:

\begin{equation}\label{Laplace-prob} {\begin{cases} \
 \Delta u  =-\lambda u , \ \ & {\rm in} \ \ \ \ \Omega, \\
  \ u=0, \ \ & {\rm on} \ \ \partial \Omega,
\end{cases}}\end{equation}where $\Delta$ denotes the Laplacian on the Riemannian manifold  $\mathcal{M}^{n}$.
Let $\lambda_{k}$ denote the $k^{th}$ eigenvalue, and then the spectrum of the eigenvalue problem \eqref{Laplace-prob} is
discrete and satisfies
\begin{equation*}
0<\lambda_{1}\leq\lambda_{2}\leq\cdots\leq\lambda_{k}\leq\cdots\rightarrow+\infty,
\end{equation*}
where each eigenvalue is repeated according to its multiplicity.
Supposing that $\Omega$ is a bounded domain on the two dimensional plane $\mathbb{R}^{2}$, eigenvalue problem \eqref{Laplace-prob} describes an interesting physical phenomenon of two-dimensional membrane vibration. For this case, Payne, P\'{o}lya and Weinberger \cite{PPW2} investigated Dirichlet problem \eqref{Laplace-prob} of Laplacian and proved that

\begin{equation}\label{ppw-2}\lambda_{2}+\lambda_{3}\leq6\lambda_{1}\end{equation}
in 1956.
Furthermore, they proposed a famous conjecture for
$\Omega\subset\mathbb{R}^{n}$ as follows:

\begin{equation}
\frac{\lambda_{2} +\lambda_{3} +\cdots+ \lambda_{n+1}}{
\lambda_{1}}\leq n\frac{\lambda_{2}(\mathbb{B}^{n})}{\lambda_{1}(\mathbb{B}^{n})},\end{equation}where $\lambda_{i}(\mathbb{B}^{n})(i=1,2)$ denotes the $i^{th}$ eigenvalue of Laplacian on  $\mathbb{B}^{n}\subset\mathbb{R}^{n}$, and $\mathbb{B}^{n}$ is an $n$-dimensional ball with the same volume as  $\Omega$, i.e., $Vol(\Omega)=Vol(\Omega^{\ast})$. Attacking this conjecture, Brands \cite{Bran} improved \eqref{ppw-2} to the following: $\lambda_{2}+\lambda_{3}  \leq\lambda_{1}(3 + \sqrt{7}),
$ when $n=2$.
Soon afterwards, Hile and Protter \cite{HP} obtained
$\lambda_{2} +\lambda_{3}
\leq 5.622\lambda_{1}.$ In 1980, Marcellini \cite{Mar} obtained $\lambda_{2}+\lambda_{3} \leq(15 + \sqrt{345})/6\lambda_{1}.$ In 2011, Chen and Zheng \cite{CZ} proved $\lambda_{2}+\lambda_{3} \leq5.3507\lambda_{1}.$ For the general dimension,  Ashbaugh and Benguria
\cite{AB1} established a universal inequality as follows:

\begin{equation}\label{1.16}
\frac{\lambda_{2} +\lambda_{3} +\cdots+ \lambda
_{n+1}}{
\lambda_{1}}\leq n + 4,\end{equation} for $\Omega\subset\mathbb{R}^{n}$ in 1993. As for further references on the solution of
this conjecture, we refer the readers to
\cite{AB2,AB3,Chit,HP,Sun}. In 2008, Sun, Cheng and Yang \cite{SCY}
studied Dirichlet problem \eqref{Laplace-prob} on the bounded domains in
a complex projective space and a unit sphere, and they derived some universal eigenvalue inequalities with lower order. In 2008, Chen and Cheng
\cite{CC} showed that inequality \eqref{1.16} remains true when $\Omega$ is a
bounded domain in a complete Riemannian manifold isometrically
minimally immersed in $\mathbb{R}^{n+p}$.
Furthermore, Ashbaugh and Benguria \cite{AB1} (cf. Hile and Protter \cite{HP} ) improved \eqref{1.16}
to

\begin{equation}\frac{\lambda_{2} + \lambda_{3}+ \cdots+ \lambda_{n+1}}{\lambda_{1}}\leq n + 3 +\frac{\lambda_{1}}{\lambda_{2}}.
\end{equation}
In 2012, Cheng and Qi \cite{CQ} proved that, for any positive integer $j$, where $1\leq j\leq n + 2$, eigenvalues satisfy at
least one of the following universal inequalities:

\begin{equation*}\begin{aligned} {\rm(1)}\ \
\frac{\lambda_{2}}{\lambda_{1}}
< 2-
\frac{\lambda_{1}}{
\lambda
_{j}}, \ \ \
 {\rm(2)}\ \ \frac{\lambda_{2} + \lambda_{3}+ \cdots+ \lambda_{n+1}}{\lambda_{1}}\leq n + 3 +\frac{\lambda_{1}}{\lambda_{j}}.
\end{aligned}\end{equation*}In 2002, Levitin and Parnovski \cite{LP} proved an abstract algebraic inequality. Applying this algebraic inequality, they  generalized  \eqref{1.16}  to
the following eigenvalue inequality: \begin{equation}\label{LP1}
\frac{\lambda_{j+1} +\lambda_{j+2} +\cdots+ \lambda_{j+n}}{
\lambda_{j}}\leq n + 4,
\end{equation}
where $j$ is any positive integer.

To describe vibrations of a clamped plate in elastic mechanics, one usually
consider the following Dirichlet problem of biharmonic operator :

\begin{equation}\label{clamped}\left\{\begin{array}{l}
\Delta^{2} u=\Lambda u, \quad \text { in } \Omega, \\
u=\frac{\partial u}{\partial \textbf{n}}=0,\quad \text { on }\partial \Omega,
\end{array}\right.\end{equation}
where $\Delta$ is the Laplacian in $\mathbb{R} ^{n}$ and $\Delta^{2}$ is the biharmonic operator in $\mathbb{R} ^{n}$, and this eigenvalue problem is called a clamped
plate problem.
In 1956, Payne, P\'{o}lya and Weinberger \cite{PPW2} also established a universal inequality for eigenvalue problem \eqref{clamped}. They obtained
\begin{equation}\label{PPW-ine-2}
\Lambda_{k+1}-\Lambda_{k} \leq \frac{8(n+2)}{n^{2}} \frac{1}{k} \sum_{i=1}^{k} \Lambda_{i}.
\end{equation}
In 1984, by improving Hile and Protter's method in \cite{HP},  Hile and Yeh \cite{HY} obtained
\begin{equation}
\sum_{i=1}^{k} \frac{\Lambda_{i}^{\frac{1}{2}}}{\Lambda_{k+1}-\Lambda_{i}} \geq \frac{n^{2} k^{3 / 2}}{8(n+2)}\left(\sum_{i=1}^{k} \Lambda_{i}\right)^{-1 / 2},
\end{equation} which generalizes the above result obtained by Payne, P\'{o}lya and Weinberger.
Furthermore, in
$1990,$ Hook \cite{Hook}, Chen and Qian \cite{CQian} proved, independently, the following inequality:

\begin{equation}\label{CQH-ineq}
\frac{n^{2} k^{2}}{8(n+2)} \leq\left[\sum_{i=1} \frac{\Lambda_{i}^{\frac{1}{2}}}{\Lambda_{k+1}-\Lambda_{i}}\right] \sum_{i=1}^{k} \Lambda_{i}^{\frac{1}{2}}.
\end{equation}
In 1999,  Ashbaugh pointed out ``\emph{whether one can establish inequalities for eigenvalues of the vibrating clamped plate problem which are analogous inequalities of Yang in the case of the eigenvalue problem of the Laplacian with Dirichlet boundary condition}" in his survey paper \cite{A1}. In 2006, Cheng and Yang \cite{CY3} gave an affirmative answer to the problem posed by  Ashbaugh. Specifically, they obtained the following:
\begin{equation}\Lambda_{k+1}-\frac{1}{k} \sum_{i=1}^{k} \Lambda_{i} \leq\left[\frac{8(n+2)}{n^{2}}\right]^{\frac{1}{2}} \frac{1}{k} \sum_{i=1}^{k}\left[\Lambda_{i}\left(\Lambda_{k+1}-\Lambda_{i}\right)\right]^{\frac{1}{2}},\end{equation}which is sharper than
\begin{equation}\label{1.11-ine}
\Lambda_{k+1} \leq\left[1+\frac{8(n+2)}{n^{2}}\right] \frac{1}{k} \sum_{i=1}^{k} \Lambda_{i}.
\end{equation}
It is easy to see that inequality \eqref{1.11-ine} is better than inequality \eqref{PPW-ine-2} given by Payne, P\'{o}lya and Weinberger. In 1999, Ashbaugh \cite{A1} announced two universal eigenvalue inequalities which are analogous to \eqref{1.16} for any dimension $n$ as follows:

\begin{equation}\sum_{\alpha=1}^{n}\left(\Lambda_{\alpha+1}^{\frac{1}{2}}-\Lambda_{1}^{\frac{1}{2}}\right) \leq 4 \Lambda_{1}^{\frac{1}{2}},\end{equation}and

\begin{equation}\sum_{\alpha=1}^{n}\left(\Lambda_{\alpha+1}-\Lambda_{1}\right) \leq 24 \Lambda_{1}.\end{equation}
Next, we consider that $X : \mathcal{M}^{n}\rightarrow \mathbb{R}^{n+p}$ is an isometric immersion from an $n$-dimensional, oriented, complete
Riemannian manifold $\mathcal{M}^{n}$ to the Euclidean space $\mathbb{R}^{n+p}$, and let $\Omega$ be a bounded domain with smooth boundary $\partial \Omega$ in $\mathcal{M}^{n}$. Assume that$\left\{e_{1}, \ldots, e_{n}\right\}$ is a local orthonormal basis of $\mathcal{M}^{n}$ with respect to the induced Riemannian metric $g$, and $\{e_{n+1}, \ldots, e_{n+p}\}$ is the corresponding local unit orthonormal normal vector fields. Assume that

$$\textbf{H} =\frac{1}{n}\sum_{\alpha=n+1}^{n+p} H^{\alpha} e_{\alpha}=\frac{1}{n}\sum_{\alpha=n+1}^{n+p}\left(\sum_{i=1}^{n} h_{i i}^{\alpha}\right) e_{\alpha},$$ is the mean curvature vector field, and

$$H=| \textbf{H} |=\frac{1}{n} \sqrt{\sum_{\alpha=n+1}^{n+p}\left(\sum_{i=1}^{n} h_{i i}^{\alpha}\right)^{2}},$$ is the mean curvature of $\mathcal{M}^{n}$ in this paper.  Let $\Pi$ denote the set of all isometric immersions from $\mathcal{M}^{n}$ into a Euclidean space. In 2010, Cheng, Huang and Wei \cite{CHW}
proved

\begin{equation}\label{chw-ineq}
\sum^{n}_{\alpha=1}(\Lambda_{\alpha+1}-\Lambda_{1})^{\frac{1}{2}}\leq4\left[\left(\frac{n}{2}+1\right)\Lambda^{\frac{1}{2}}_{1}+C_{0}
\right]^{\frac{1}{2}}\left(\Lambda^{\frac{1}{2}}_{1}+C_{0}\right)^{\frac{1}{2}},
\end{equation} where $$
C_{0}=\frac{1}{4}\inf _{\sigma \in \Pi}\max_{\Omega}\left(n^{2}H^{2}\right).
$$In particular, when $\mathcal{M}^{n}$ is an $n$-dimensional
complete minimal submanifold in a Euclidean space, \eqref{chw-ineq} implies

\begin{equation}\label{chw-ineq-0}\sum^{n}_{\alpha=1}(\Lambda_{\alpha+1}-\Lambda_{1})^{\frac{1}{2}}\leq[8(n+2)\Lambda_{1}]^{\frac{1}{2}} .
\end{equation}
In 2011, Wang and Xia \cite{WX3} investigated the eigenvalues with higher order of bi-harmonic operator on the complete Riemannian manifolds and proved the following inequality:

\begin{equation}\begin{aligned}
\sum_{i=1}^{k}\left(\Lambda_{k+1}-\Lambda_{i}\right)^{2} \leq & \frac{4}{n}\left\{\sum_{i=1}^{k}\left(\Lambda_{k+1}-\Lambda_{i}\right)^{2}\left[\left(\frac{n}{2}+1\right)\Lambda_{i}^{\frac{1}{2}}+C_{0} \right]\right\}^{\frac{1}{2}} \\
& \times\left\{\sum_{i=1}^{k}\left(\Lambda_{k+1}-\Lambda_{i}\right)\left(\Lambda_{i}^{\frac{1}{2}}+C_{0}\right)\right\}^{\frac{1}{2}},
\end{aligned}\end{equation}where
$$
C_{0}=\frac{1}{4}\inf _{\sigma \in \Pi}\max_{\Omega}\left(n^{2}H^{2}\right).
$$
Let $\nu$ be a constant vector field defined on $\mathcal{M}^{n}$.
Throughout this paper, we use $\langle\cdot,\cdot\rangle_{g}$, $|\cdot|_{g}^{2}$, ${\rm div}$, $\Delta$, $\nabla$ and $\nu^{\top}$ to denote the Riemannian inner product with respect to the induced metric $g$, norm associated with the inner product $\langle\cdot,\cdot\rangle_{g}$,  divergence, Laplacian, the gradient operator on $\mathcal{M}^{n}$ and the projective of the vector $\nu$ on the tangent bundle $T\mathcal{M}^{n}$, respectively. Next, we define an elliptic operator on $\mathcal{M}^{n}$ as follows:

\begin{equation} \label{L-v} \mathfrak{L}_{\nu} =\Delta+ \langle\nu,\nabla(\cdot)\rangle_{g_{0}}=e^{-\langle\nu,X\rangle_{g_{0}}}{\rm div}(e^{\langle\nu,X\rangle_{g_{0}}}\nabla(\cdot)),\end{equation}
where $\langle\cdot, \cdot\rangle_{g_{0}}$ denotes the standard inner product of $\mathbb{R}^{n+p}$. Correspondingly, we use $|\cdot|_{g_{0}}$ to denote the norm on $\mathbb{R}^{n+p}$ associated with the standard inner product $\langle\cdot,\cdot\rangle_{g_{0}}$. In particular, we assume that $\nu$ is a unit constant vector defined on a translating soliton in the sense of the means curvature flows and denote it by $\nu_{0}$. Then, the differential operator will be denoted by $\mathfrak{L}_{II}$, which is introduced by Xin in \cite{Xin2} and of important geometric meaning. We refer the readers to section \ref{sec5} for details.

\begin{rem}It can be shown that the elliptic differential operator $\mathfrak{L}_{\nu}$ is a self-adjoint operator with
respect to the weighted measure $e^{\langle\nu,X\rangle_{g_{0}}}dv$.  Namely, for any $u, w
\in C_{1}^{2}(\Omega)$, the following formula holds:
\begin{equation}
\begin{aligned}
\label{1.3} -\int_{\Omega}\langle\nabla u,\nabla w\rangle_{g} e^{\langle\nu,X\rangle_{g_{0}}}dv
=\int_{\Omega}(\mathfrak{L}_{\nu}w)ue^{\langle\nu,X\rangle_{g_{0}}}dv=\int_{\Omega}(\mathfrak{L}_{\nu}u)we^{\langle\nu,X\rangle_{g_{0}}}dv.
\end{aligned}
\end{equation}\end{rem}

Just like the other weighted Laplacian, for example, $\mathcal{L}$ operator and Witten-Laplacian, $\mathcal{L}_{\nu}$ operator is also very important in geometric analysis. However, the eigenvalues of such an operator are rarely studied as far as we know. Therefore, it is very
urgent for us to exploit the eigenvalue problem of $\mathcal{L}_{\nu}$ operator. Next, let us consider an eigenvalue problem of $\mathfrak{L}_{\nu}^{2}$ operator on the bounded domain $\Omega \subset \mathcal{M}^{n}$ with Dirichlet boundary condition:

\begin{equation}\label{L-2-prob} {\begin{cases} \
 \mathfrak{L}_{\nu}^{2}u  =\Lambda u , \ \ & {\rm in} \ \ \ \ \Omega, \\
  \ u=\frac{\partial u}{\partial \textbf{n}}=0, \ \ & {\rm on} \ \ \partial \Omega,
\end{cases}}\end{equation}where $\textbf{n}$ denotes the normal vector to the boundary $\partial\Omega$.
The main goal of this paper is to establish some eigenvalue inequalities with lower order for  clamped plate problem \eqref{L-2-prob}
of $\mathfrak{L}_{\nu}^{2}$ operator on $\mathcal{M}^{n}$. However, for the eigenvalues with higher order of $\mathcal{L}_{\nu}$ operator and $\mathcal{L}^{2}_{\nu}$ operator, we also obtain
some eigenvalue inequalities in some separated papers, elsewhere.
Now, let us state the main result as follows.

\begin{thm}
\label{thm1.1}
Let $\mathcal{M}^{n}$ be an $n$-dimensional complete Riemannian
manifold isometrically embedded into the Euclidean space $\mathbb{R}^{n+p}$ with mean curvature $H$. Then,  for any $j$, where $j=1,2,\cdots$, eigenvalues of clamped plate problem \eqref{L-2-prob}
  satisfy
\begin{equation}\begin{aligned}\label{z-ineq-1}
&\sum_{i=1}^{n}\left(\Lambda_{i+1}-\Lambda_{1}\right)^{\frac{1}{2}}\\&\leq4\left\{\left( \Lambda_{1}^{\frac{1}{2}}+4\widetilde{C}_{1}\Lambda_{1}^{\frac{1}{4}}+4\widetilde{C}_{1}^{2}+C_{1}\right)\left[\left(\frac{n}{2}+1\right) \Lambda_{1}^{\frac{1}{2}}+4\widetilde{C}_{1}\Lambda_{1}^{\frac{1}{4}}+4\widetilde{C}_{1}^{2}+C_{1}\right]\right\}^{\frac{1}{2}},
\end{aligned}\end{equation}
where $C_{1}$ is given by

$$
C_{1}=\frac{1}{4}\inf _{\sigma \in \Pi}\max_{\Omega}\left(n^{2}H^{2}\right),
$$and $\widetilde{C}_{1}$ is given by $$\widetilde{C}_{1}=\frac{1}{4}\max_{\Omega} |\nu^{\top}|_{g_{0}} .$$
\end{thm}

\begin{rem}In theorem {\rm \ref{thm1.1}}, if $|\nu|_{g_{0}}=0$, then one can deduce \eqref{chw-ineq} from \eqref{z-ineq-1}.\end{rem}

\begin{corr}\label{corr1.1}Under the assumption of theorem {\rm \ref{thm1.1}}, we have
\begin{equation*}
\sum_{i=1}^{n}\left\{\left(\Lambda_{i+1}-\Lambda_{1}\right)^{\frac{1}{2}}-\Lambda_{1}^{\frac{1}{2}}\right\} \leq4\left(\Lambda_{1}^{\frac{1}{2}}+4\widetilde{C}_{1}\Lambda_{1}^{\frac{1}{4}}+4\widetilde{C}_{1}^{2}+C_{1}\right),
\end{equation*}where $C_{1}$ is given by

$$
C_{1}=\frac{1}{4}\inf _{\sigma \in \Pi}\max_{\Omega}\left(n^{2}H^{2}\right),
$$and $\widetilde{C}_{1}$ is given by $$\widetilde{C}_{1}=\frac{1}{4}\max_{\Omega} |\nu^{\top}|_{g_{0}}.$$
 \end{corr}

\begin{corr}
\label{corr1.2}
Under the assumption of theorem {\rm \ref{thm1.1}}, we have
\begin{equation}\begin{aligned}\label{z-ineq-1-1}
\sum_{i=1}^{n}\left(\Lambda_{i+1}-\Lambda_{1}\right)^{\frac{1}{2}}\leq6\left\{\left( \Lambda_{1}^{\frac{1}{2}}+C_{2}\right)\left[\left(\frac{n}{3}+1\right) \Lambda_{1}^{\frac{1}{2}}+C_{2}\right]\right\}^{\frac{1}{2}},
\end{aligned}\end{equation}
where $C_{2}$ is given by

$$
C_{2}=\frac{1}{6}\inf _{\sigma \in \Pi}\max_{\Omega}\left(n^{2}H^{2}+3|\nu^{\top}|_{g_{0}}^{2}\right).
$$
\end{corr}

\begin{corr}\label{corr1.3}Under the assumption of theorem {\rm \ref{thm1.1}}, we have
\begin{equation*}
\sum_{i=1}^{n}\left\{\left(\Lambda_{i+1}-\Lambda_{1}\right)^{\frac{1}{2}}-\Lambda_{1}^{\frac{1}{2}}\right\} \leq 6 \left(\Lambda_{1}^{\frac{1}{2}}+C_{2}\right),
\end{equation*}where $C_{2}$ is given by

$$
C_{2}=\frac{1}{6}\inf _{\sigma \in \Pi}\max_{\Omega}\left(n^{2}H^{2}+3|\nu^{\top}|_{g_{0}}^{2}\right).
$$\end{corr}
This paper is organized as follows. In Section \ref{sec2}, we prove a general formula for
eigenvalues of the clamped plate problem of  $\mathfrak{L}_{\nu}^{2}$ operator on the complete Riemannian manifolds. In section \ref{sec3}, we prove some results of Chen and Cheng type, which will be very useful in the proof of our main reuslts. By making use
of the general formula and some results of Chen-Cheng type, we give the proofs of theorem \ref{thm1.1}, corollary \ref{corr1.1}, corollary \ref{corr1.2} and corollary \ref{corr1.3} in section \ref{sec4}. Applying theorem \ref{thm1.1}, we obtain several universal inequalities for the eigenvalues of $\mathfrak{L}_{II}^{2}$ operator on the translating solitons with respect to the mean curvature flows in section \ref{sec5}. We note that all of eigenvalue inequalities of $\mathfrak{L}_{II}^{2}$ operator on the translating solitons are universal. Finally, as some further applications of theorem \ref{thm1.1},  we obtain some eigenvalue inequalities of $\mathfrak{L}_{\nu}^{2}$ operator on the minimal submanifolds isometrically embedded into the Euclidean spaces, submanifolds isometrically embedded (or immersed) into the unit spheres and projective spaces in section \ref{sec6}.

\begin{ack}The research was partially supported by the National Natural
Science Foundation of China (Grant Nos. 11861036 and 11826213) and  the  Natural Science Foundation of Jiangxi Province (Grant No. 20171ACB21023). The author  shall express his sincere gratitude to  the anonymous referees for
their helpful comments and suggestions.
\end{ack}

\section{Some Lemmas and their Proofs} \label{sec2}
\vskip3mm
In this section, we establish a general formula, which will play an important role in the proof of theorem \ref{thm1.1}.

Assume that $\zeta\in C^{2}_{0}(\Omega)$, and we define $\xi$ as follows:

\begin{equation}\begin{aligned}\label{xi}\xi :& =u_{1} \mathfrak{L}_{\nu}^{2}\zeta  + 2\langle \nabla u_{1}, \nabla ( \mathfrak{L}_{\nu} \zeta )\rangle_{g}  + 2 \mathfrak{L}_{\nu} \zeta  \mathfrak{L}_{\nu} u_{1} \\&\quad+ 2 \mathfrak{L}_{\nu}\langle \nabla \zeta , \nabla u_{1}\rangle_{g}
 +2\langle \nabla \zeta , \nabla ( \mathfrak{L}_{\nu} u_{1})\rangle_{g}.\end{aligned}\end{equation}
Then, we have the following lemma.

\begin{lem}\label{lem2.1}
\label{lem2.1-1}Let $(\mathcal{M}^{n},g)$ be an $n$-dimensional complete Riemannian manifold and $u_{1}$  the first eigenfunction of the eigenvalue problem
\eqref{L-2-prob}. Then, we have

\begin{equation}\label{zero-15}\int
_{\Omega}\xi u_{1}e^{\langle\nu,X\rangle_{g_{0}}}dv = 0.\end{equation}\end{lem}

\begin{proof}By the definition of $\xi$ given by \eqref{xi}, we know that

\begin{equation}\begin{aligned}\label{le-eq1}\int
_{\Omega}\xi u_{1}e^{\langle\nu,X\rangle_{g_{0}}}dv &=\int
_{\Omega}u_{1}^{2} \mathfrak{L}_{\nu}^{2}\zeta e^{\langle\nu,X\rangle_{g_{0}}}dv+2\int
_{\Omega}u_{1}\langle \nabla u_{1}, \nabla ( \mathfrak{L}_{\nu} \zeta )\rangle_{g} e^{\langle\nu,X\rangle_{g_{0}}}dv\\&\quad+2\int
_{\Omega}u_{1} \mathfrak{L}_{\nu} \zeta  \mathfrak{L}_{\nu} u_{1}e^{\langle\nu,X\rangle_{g_{0}}}dv+2\int
_{\Omega}u_{1} \mathfrak{L}_{\nu}\langle \nabla \zeta , \nabla u_{1}\rangle_{g} e^{\langle\nu,X\rangle_{g_{0}}}dv\\&\quad+2\int
_{\Omega}u_{1}\langle \nabla \zeta , \nabla ( \mathfrak{L}_{\nu} u_{1})\rangle_{g} e^{\langle\nu,X\rangle_{g_{0}}}dv.
\end{aligned}\end{equation}
By Stokes's theorem, we have

\begin{equation}\begin{aligned}\label{le-eq2}
&\int
_{\Omega}u_{1}^{2} \mathfrak{L}_{\nu}^{2}\zeta e^{\langle\nu,X\rangle_{g_{0}}}dv+2\int
_{\Omega}u_{1}\langle \nabla u_{1}, \nabla ( \mathfrak{L}_{\nu} \zeta )\rangle_{g} e^{\langle\nu,X\rangle_{g_{0}}}dv
\\&=\int
_{\Omega}u_{1}^{2} \mathfrak{L}_{\nu}^{2}\zeta e^{\langle\nu,X\rangle_{g_{0}}}dv+\int
_{\Omega}\langle \nabla u_{1}^{2}, \nabla ( \mathfrak{L}_{\nu} \zeta )\rangle_{g} e^{\langle\nu,X\rangle_{g_{0}}}dv\\&=\int
_{\Omega}u_{1}^{2} \mathfrak{L}_{\nu}^{2}\zeta e^{\langle\nu,X\rangle_{g_{0}}}dv-\int
_{\Omega}u_{1}^{2} \mathfrak{L}_{\nu}^{2}\zeta e^{\langle\nu,X\rangle_{g_{0}}}dv\\&=0.
\end{aligned}\end{equation}
Applying Stokes's theorem again, we infer that

\begin{equation}\begin{aligned}\label{le-eq3}
0&=\int_{\Omega}{\rm div}\left(u_{1} \mathfrak{L}_{\nu} u_{1}e^{\langle\nu,X\rangle_{g_{0}}}\nabla \zeta \right)dv\\
&=\int_{\Omega}\langle\nabla u_{1},\nabla \zeta \rangle_{g} \mathfrak{L}_{\nu} u_{1}e^{\langle\nu,X\rangle_{g_{0}}}dv+
\int_{\Omega}u_{1}\langle\nabla( \mathfrak{L}_{\nu} u_{1}),\nabla \zeta \rangle_{g}
 e^{\langle\nu,X\rangle_{g_{0}}}dv\\&\quad+
\int_{\Omega} u_{1} \mathfrak{L}_{\nu} u_{1}\langle\nu,\nabla \zeta \rangle_{g_{0}} e^{\langle\nu,X\rangle_{g_{0}}}dv+\int_{\Omega}u_{1}\mathfrak{L}_{\nu}
 u_{1}\Delta\zeta e^{\langle\nu,X\rangle_{g_{0}}
 }dv
\\&=\int_{\Omega}\langle\nabla u_{1},\nabla \zeta \rangle_{g} \mathfrak{L}_{\nu} u_{1}e^{\langle\nu,X\rangle_{g_{0}}}dv+
\int_{\Omega}u_{1}\langle\nabla( \mathfrak{L}_{\nu} u_{1}),\nabla \zeta \rangle_{g} e^{\langle\nu,X\rangle_{g_{0}}}dv\\&\quad\quad+\int_{\Omega}u_{1} \mathfrak{L}_{\nu} u_{1} \mathfrak{L}_{\nu}\zeta e^{\langle\nu,X\rangle_{g_{0}}}dv
.
\end{aligned}\end{equation}
From \eqref{le-eq1}, \eqref{le-eq2} and \eqref{le-eq3}, we derive \eqref{zero-15}.
 Hence, we finish the proof of this lemma.

\end{proof}
\begin{lem} \label{lemma2.2}\textbf{\emph{(General formula)}} Let $\Omega$ be a bounded domain on an $n$-dimensional complete Riemannian submanifold $(\mathcal{M}^{n},g)$ isometrically immersed  into the Euclidean space
$\mathbb{R}^{n+p}$, and $\Lambda_{i}$ be the $i^{\text {th}}$ eigenvalue of the eigenvalue problem \eqref{L-2-prob} and $u_{i}$ be the orthonormal eigenfunction corresponding to $\Lambda_{i},$ that is,
$$
\left\{\begin{array}{ll}
\mathfrak{L}_{\nu}^{2} u_{i}=\Lambda_{i} u_{i}, & \text { in } \Omega \\
u=\frac{\partial u}{\partial \textbf{n}}=0, & \text { on } \partial \Omega \\
\int_{\Omega} u_{i} u_{j} e^{\langle\nu,X\rangle_{g_{0}}}dv=\delta_{i j}, & \forall i, j=1,2, \ldots
\end{array}\right.
$$
where $\textbf{n}$ is an outward normal vector field of $\partial \Omega$.
If $\phi_{i}(i \geq 2) \in C^{4}(\Omega) \cap C^{3}(\partial \Omega)$ satisfies $\int_{\Omega} \phi_{i} u_{1} u_{j+1}e^{\langle\nu,X\rangle_{g_{0}}}dv=0$ for $1 \leq j<i,$ then for any positive integer $i,$ we have

\begin{equation}\begin{aligned}\label{lem-2.2}
\left(\Lambda_{i+1}-\Lambda_{1}\right)^{\frac{1}{2}} \int_{\Omega}\left|u_{1} \nabla \phi_{i}\right|^{2}_{g} e^{\langle\nu,X\rangle_{g_{0}}}
dv&\leq\left(\frac{\delta}{2}+\frac{1}{2 \delta}\right)\int_{\Omega}\Upsilon(\phi_{i})e^{\langle\nu,X\rangle_{g_{0}}}dv \\&\quad\quad-\delta \int_{\Omega}\Phi(\phi_{i})e^{\langle\nu,X\rangle_{g_{0}}}
dv,
\end{aligned}\end{equation}
where

\begin{equation}\label{Upsilon}\Upsilon(\phi_{i})=\left(u_{1} \mathfrak{L}_{\nu} \phi_{i}+2\langle \nabla \phi_{i},\nabla u_{1}\rangle_{g}\right)^{2},\end{equation}

\begin{equation}\label{Phi}\Phi(\phi_{i})=\left|\nabla \phi_{i}\right|^{2}_{g} u_{1} \mathfrak{L}_{\nu} u_{1},\end{equation}and
$\delta$ is any positive constant.
\end{lem}

\begin{proof}
In order to prove \eqref{lem-2.2}, let us define

\begin{equation}
\psi_{i}:=\left(\phi_{i}-a_{i}\right) u_{1},
\end{equation}
where $i\geq2$ and $$a_{i}=\int_{\Omega} \phi_{i} u_{1}^{2} e^{\langle\nu,X\rangle_{g_{0}}}dv.$$ It is not difficult to check that

$$\int_{\Omega} \psi_{i} u_{1} e^{\langle\nu,X\rangle_{g_{0}}}dv=0 .$$ Noticing $$\int_{\Omega} \phi_{i} u_{1} u_{j+1} e^{\langle\nu,X\rangle_{g_{0}}
}dv=0 \ \ {\rm for}\ \ 1 \leq j<i,$$ we infer

$$
\int_{\Omega} \psi_{i} u_{j+1} e^{\langle\nu,X\rangle_{g_{0}}}dv=0, \text { for } 1 \leq j<i,$$
and

$$ \psi_{i}\bigg|_{\partial \Omega}=\frac{\partial \psi_{i}}{\partial \nu}\bigg|_{\partial \Omega}=0.
$$
From the Rayleigh-Ritz inequality, we have

\begin{equation}\label{RR-2}
\Lambda_{i+1} \int_{\Omega} \psi_{i}^{2} e^{\langle\nu,X\rangle_{g_{0}}}
dv \leq \int_{\Omega} \psi_{i} \mathfrak{L}_{\nu}^{2} \psi_{i} e^{\langle\nu,X\rangle_{g_{0}}}dv.
\end{equation}
According to the definition of the function $\psi_{i},$ one has
$$
\begin{aligned}
\mathfrak{L}_{\nu}\left(\psi_{i}\right) = \mathfrak{L}_{\nu}\left(\phi_{i} u_{1}\right)-a_{i} \mathfrak{L}_{\nu} u_{1}
=u_{1} \mathfrak{L}_{\nu}\phi_{i}+2\left\langle\nabla \phi_{i}, \nabla u_{1}\right\rangle_{g}+\phi_{i} \mathfrak{L}_{\nu} u_{1}-a_{i} \mathfrak{L}_{\nu} u_{1},
\end{aligned}
$$
and

\begin{equation*}
\begin{aligned}
\mathfrak{L}_{\nu}^{2}\left(\psi_{i}\right) &= \mathfrak{L}_{\nu}\left(\mathfrak{L}_{\nu}\left(\psi_{i}\right) \right) \\
&= \mathfrak{L}_{\nu}\left(u_{1} \mathfrak{L}_{\nu} \phi_{i}+2\left\langle\nabla \phi_{i}, \nabla u_{1}\right\rangle_{g}+\phi_{i} \mathfrak{L}_{\nu} u_{1}-a_{i} \mathfrak{L}_{\nu} u_{1}\right)\\&= u_{1} \mathfrak{L}_{\nu}^{2} \phi_{i}+2\left\langle\nabla u_{1}, \nabla\left( \mathfrak{L}_{\nu} \phi_{i}\right)\right\rangle_{g}+2 \mathfrak{L}_{\nu} \phi_{i} \mathfrak{L}_{\nu} u_{1}+2 \mathfrak{L}_{\nu}\left\langle\nabla \phi_{i}, \nabla u_{1}\right\rangle_{g} \\
&+2\left\langle\nabla \phi_{i}, \nabla\left( \mathfrak{L}_{\nu} u_{1}\right)\right\rangle_{g}+\phi_{i} \mathfrak{L}_{\nu}^{2} u_{1}-a_{i} \mathfrak{L}_{\nu}^{2} u_{1} \\&
= \tau_{i}+\Lambda_{1} \psi_{i},
\end{aligned}
\end{equation*}where
$$
\begin{aligned}
\epsilon_{i}= u_{1} \mathfrak{L}_{\nu}^{2} \phi_{i}+2\left\langle\nabla u_{1}, \nabla\left( \mathfrak{L}_{\nu} \phi
_{i}\right)\right\rangle_{g}+2 \mathfrak{L}_{\nu} \phi_{i} \mathfrak{L}_{\nu} u_{1}
&+2 \mathfrak{L}_{\nu}\left\langle\nabla \phi_{i}, \nabla u_{1}\right\rangle_{g}+2\left\langle\nabla \phi
_{i}, \nabla\left( \mathfrak{L}_{\nu} u_{1}\right)\right\rangle_{g}.
\end{aligned}
$$From \eqref{RR-2}, we conclude that

\begin{equation}\begin{aligned}\label{RR-2-1}
\left(\Lambda_{i+1}-\Lambda_{1}\right) \int_{\Omega} \psi_{i}^{2} e^{\langle\nu,X\rangle_{g_{0}}}
dv &\leq \int_{\Omega} \psi_{i} \epsilon_{i} e^{\langle\nu,X\rangle_{g_{0}}
}dv\\&=\int_{\Omega} \epsilon_{i} \phi_{i} u_{1} e^{\langle\nu,X\rangle_{g_{0}}}dv-a_{i} \int_{\Omega} \epsilon_{i} u_{1} e^{\langle\nu,X\rangle_{g_{0}}}dv.
\end{aligned}\end{equation}
By lemma \ref{xi}, we know that

\begin{equation}\label{tau-u1=0}
\int_{\Omega} \epsilon_{i} u_{1} e^{\langle\nu,X\rangle_{g_{0}}}dv=0.
\end{equation}
Applying  Stokes' theorem, we have the following equalities:

\begin{equation}\label{inea-1}
\begin{aligned}&
2 \int_{\Omega} \phi_{i} u_{1}\left\langle\nabla u_{1}, \nabla\left( \mathfrak{L}_{\nu} \phi_{i}\right)\right\rangle_{g}
 e^{\langle\nu,X\rangle_{g_{0}}}dv \\&
\quad=\int_{\Omega}\left(2 u_{1} \mathfrak{L}_{\nu} \phi_{i}\left\langle\nabla u_{1}, \nabla \phi_{i}\right\rangle_{g}+u_{1}^{2}\left( \mathfrak{L}_{\nu} \phi_{i}\right)^{2}-\phi_{i} u_{1}^{2} \mathfrak{L}_{\nu}^{2} \phi_{i}\right) e^{\langle\nu,X\rangle_{g_{0}}}dv,\end{aligned}
\end{equation}

\begin{equation}\label{inea-2}
\begin{aligned}&
2 \int_{\Omega} \phi_{i} u_{1} \mathfrak{L}_{\nu}\left\langle\nabla \phi_{i}, \nabla u_{1}\right\rangle_{g} e^{\langle\nu,X\rangle_{g_{0}}}dv \\&
\quad=\int_{\Omega}\left(2 \mathfrak{L}_{\nu} \phi_{i} u_{1}\left\langle\nabla \phi_{i}, \nabla u_{1}\right\rangle_{g}+4\left\langle\nabla \phi_{i}, \nabla u_{1}\right\rangle_{g}^{2}+2 \phi_{i} \mathfrak{L}_{\nu} u_{1}\left\langle\nabla \phi_{i}, \nabla u_{1}\right\rangle_{g}\right) e^{\langle\nu,X\rangle_{g_{0}}}dv,\end{aligned}
\end{equation}
and

\begin{equation}\label{inea-3}
\begin{aligned}
&2 \int_{\Omega} \phi_{i} u_{1}\left\langle\nabla \phi_{i}, \nabla\left( \mathfrak{L}_{\nu} u_{1}\right)\right\rangle_{g} e^{\langle\nu,X\rangle_{g_{0}}}dv \\&
\quad=-2 \int_{\Omega}\left(\left|\nabla \phi_{i}\right|^{2}_{g} u_{1} \mathfrak{L}_{\nu} u_{1}+\phi_{i}\mathfrak{L}_{\nu} u_{1}\left\langle\nabla \phi_{i}, \nabla u_{1}\right\rangle_{g}+\phi_{i} u_{1} \mathfrak{L}_{\nu} \phi_{i} \mathfrak{L}_{\nu} u_{1}\right) e^{\langle\nu,X\rangle_{g_{0}}}dv.
\end{aligned}
\end{equation}
Combining \eqref{inea-1}-\eqref{inea-3}, we infer that

\begin{equation}
\begin{aligned}\label{th-i-ineq}
&\int_{\Omega} \epsilon_{i} \phi_{i} u_{1} e^{\langle\nu,X\rangle_{g_{0}}}dv\\&= \int_{\Omega}\left(\left( \mathfrak{L}_{\nu} \phi_{i}\right)^{2} u_{1}^{2}+4\left\langle\nabla \phi_{i}, \nabla u_{1}\right\rangle_{g}^{2}+4 u_{1}  \mathfrak{L}_{\nu} \phi_{i}\left\langle\nabla \phi_{i}, \nabla u_{1}\right\rangle_{g}\right) e^{\langle\nu,X\rangle_{g_{0}}}dv \\
&\quad\quad -\int_{\Omega} 2\left|\nabla \phi_{i}\right|^{2}_{g} u_{1} \mathfrak{L}_{\nu} u_{1} e^{\langle\nu,X\rangle_{g_{0}}}dv
\\&
=\int_{\Omega}\left(u_{1} \mathfrak{L}_{\nu} \phi_{i}+2 \left\langle\nabla \phi_{i}, \nabla u_{1}\right\rangle_{g}\right)^{2} e^{\langle\nu,X\rangle_{g_{0}}}dv-\int_{\Omega} 2\left|\nabla \phi_{i}\right|_{g}^{2} u_{1} \mathfrak{L}_{\nu} u_{1} e^{\langle\nu,X\rangle_{g_{0}}}dv.
\end{aligned}
\end{equation}
Substituting  \eqref{tau-u1=0} and \eqref{th-i-ineq} into \eqref{RR-2-1}, one can conclude that

\begin{equation*}
\begin{aligned}
\left(\Lambda_{i+1}-\Lambda_{1}\right) \int_{\Omega} \psi_{i}^{2} e^{\langle\nu,X\rangle_{g_{0}}}dv &\leq \int_{\Omega}\left(u_{1} \mathfrak{L}_{\nu} \phi_{i}+2 \left\langle\nabla \phi_{i}, \nabla u_{1}\right\rangle_{g}
\right)^{2} e^{\langle\nu,X\rangle_{g_{0}}}dv\\&\quad -\int_{\Omega} 2\left|\nabla \phi_{i}\right|_{g}^{2} u_{1} \mathfrak{L}_{\nu} u_{1} e^{\langle\nu,X\rangle_{g_{0}}}dv.
\end{aligned}
\end{equation*}
On the other hand, we have$$\begin{aligned}
\int_{\Omega} & \psi_{i}\left(u_{1} \mathfrak{L}_{\nu} \phi
_{i}+2\langle \nabla \phi_{i},\nabla u_{1}\rangle_{g}\right) e^{\langle\nu,X\rangle_{g_{0}}}dv \\
&=\int_{\Omega}\left(\phi_{i}-a_{i}\right) u_{1}\left(u_{1} \mathfrak{L}_{\nu} \phi_{i}+2\langle \nabla \phi_{i},\nabla u_{1}\rangle_{g}\right) e^{\langle\nu,X\rangle_{g_{0}}}dv \\
&=\int_{\Omega} \phi
_{i} u_{1}\left(u_{1} \mathfrak{L}_{\nu} \phi_{i}+2\langle \nabla \phi_{i},\nabla u_{1}\rangle_{g}\right) e^{\langle\nu,X\rangle_{g_{0}}}dv \\
&=\int_{\Omega}\left(\phi
_{i} u_{1}^{2} \mathfrak{L}_{\nu} \phi_{i}+2 \phi_{i} u_{1}\left\langle\nabla \phi_{i}, \nabla u_{1}\right\rangle_{g}\right) e^{\langle\nu,X\rangle_{g_{0}}}dv \\
&=-\int_{\Omega}\left|u_{1} \nabla \phi_{i}\right|^{2}_{g} e^{\langle\nu,X\rangle_{g_{0}}}dv.
\end{aligned}$$
So, for any $\delta>0,$ we have
\begin{equation*}\begin{aligned}&
\left(\Lambda_{i+1}-\Lambda_{1}\right)^{\frac{1}{2}} \int_{\Omega}\left|u_{1} \nabla \phi_{i}\right|_{g}^{2} e^{\langle\nu,X\rangle_{g_{0}}}dv \\&
\quad=\left(\Lambda_{i+1}-\Lambda_{1}\right)^{\frac{1}{2}} \int_{\Omega}-\psi_{i}\left(u_{1} \mathfrak{L}_{\nu} \phi_{i}+2\langle \nabla \phi_{i},\nabla u_{1}\rangle_{g}\right) e^{\langle\nu,X\rangle_{g_{0}}}dv \\&
\quad \leq \frac{\delta}{2}\left(\Lambda_{i+1}-\Lambda_{1}\right) \int_{\Omega} \psi_{i}^{2}e^{\langle\nu,X\rangle_{g_{0}}}dv+\frac{1}{2 \delta} \int_{\Omega}\left(u_{1} \mathfrak{L}_{\nu} \phi_{i}+2\langle \nabla \phi_{i},\nabla u_{1}\rangle_{g}\right)^{2} e^{\langle\nu,X\rangle_{g_{0}}}dv \\&
\quad \leq\left(\frac{\delta}{2}+\frac{1}{2 \delta}\right) \int_{\Omega}\left(u_{1} \mathfrak{L}_{\nu} \phi_{i}+2\langle \nabla \phi_{i},\nabla u_{1}\rangle_{g}\right)^{2} e^{\langle\nu,X\rangle_{g_{0}}}dv-\delta \int_{\Omega}\left|\nabla \phi
_{i}\right|_{g}^{2} u_{1} \mathfrak{L}_{\nu} u_{1} e^{\langle\nu,X\rangle_{g_{0}}}dv,
\end{aligned}\end{equation*}
which means that \eqref{lem-2.2} is true. This completes the proof of Lemma \ref{lemma2.2}.

\end{proof}

\section{Some Results of Chen-Cheng Type}\label{sec3}

\vskip3mm \noindent In order to prove our main results, the
following lemmas will play  very important roles. The first lemma reads as follows:

\begin{lem}\label{lem3.1}Let $\mathcal{M}^{n}$ be
   an $n$-dimensional  submanifold  in Euclidean space
$\mathbb{R}^{n+p}$,  and $y=(y^{1},y^{2},\cdots,y^{n+p})$ be the
position vector of a point $p\in \mathcal{M}^{n}$ with
$y^{\alpha}=y^{\alpha}(x^{1}, \cdots,x^{n})$, $1\leq \alpha\leq
n+p$, where $(x^{1}, \cdots, x^{n})$ denotes a local coordinate
system of $\mathcal{M}^n$. Then, we have
\begin{equation}\label{n-ine}
\sum^{n+p}_{\alpha=1}\langle\nabla y^{\alpha},\nabla y^{\alpha}\rangle_{g}= n,
\end{equation}
\begin{equation}
\begin{aligned}\label{uw-ine}
\sum^{n+p}_{\alpha=1}\langle\nabla y^{\alpha},\nabla u\rangle_{g}\langle\nabla
y^{\alpha},\nabla w\rangle_{g}=\langle\nabla u,\nabla w\rangle_{g},
\end{aligned}
\end{equation}
for any functions  $u, w\in C^{1}(\mathcal{M}^{n})$,
\begin{equation}
\begin{aligned}\label{nH-ine}
\sum^{n+p}_{\alpha=1}(\Delta y^{\alpha})^{2}=n^{2}H^{2},
\end{aligned}
\end{equation}
\begin{equation}
\begin{aligned}\label{0-ine}
\sum^{n+p}_{\alpha=1}\Delta y^{\alpha}\nabla y^{\alpha}= \textbf{0},
\end{aligned}
\end{equation}
where $H$ is the mean curvature of $\mathcal{M}^{n}$.
\end{lem}

 A proof of lemma \ref{lem3.1} can be
found in \cite{CC}. Or see \cite{CHW}. Similarly, we have the following lemma.
\begin{lem}\label{lem3.2}
Let $\left(x^{1}, \cdots, x^{n}\right)$ be an arbitrary coordinate system in a neighborhood $U$ of $P$ in $\mathcal{M}^{n} .$ Assume that $y$ with components $y^{\alpha}$ defined by
$
y^{\alpha}=y^{\alpha}\left(x^{1}, \cdots, x^{n}\right),  1 \leq \alpha \leq n+p,
$
is the position vector of $P$ in $\mathbb{R} ^{n+p}$.
Then, we have
\begin{equation}\label{v-2}
\sum_{\alpha=1}^{n+p}\left\langle\nabla y^{\alpha}, \nu\right\rangle_{g_{0}}^{2}=|\nu^{\top}|_{g_{0}}^{2},\end{equation}
where $\nabla$ is the gradient operator on $\mathcal{M}^{n}$.

\end{lem}
\begin{proof} Equality \eqref{v-2} can be proved as follows:$$\begin{aligned}
\sum_{\alpha=1}^{n+p}\langle\nabla y^{\alpha}, v\rangle_{g_{0}}^{2}&=\sum_{\alpha=1}^{n+p}\langle\nabla y^{\alpha}, \nu^{\top}\rangle_{g_{0}}^{2}
=\sum_{\alpha=1}^{n+p}\left(\nu^{\top} y^{\alpha}\right)^{2}=\left|\nu^{\top}\right|_{g_{0}}^{2}.
\end{aligned}
$$
  Therefore, it finishes the proof of lemma \ref{lem3.2}.\end{proof}

\begin{lem}\label{lem3.3}
Let $\left(x^{1}, \cdots, x^{n}\right)$ be an arbitrary coordinate system in a neighborhood $U$ of $P$ in $\mathcal{M}^{n} .$ Assume that $y$ with components $y^{\alpha}$ defined by
$
y^{\alpha}=y^{\alpha}\left(x^{1}, \cdots, x^{n}\right), 1 \leq \alpha \leq n+p,
$
is the position vector of $P$ in $\mathbb{R} ^{n+p}$.
Then, we have

\begin{equation}\label{uv-ineq}
\sum_{\alpha=1}^{n+p}\left\langle\nabla y^{\alpha}, \nabla u\right\rangle_{g}\left\langle\nabla y^{\alpha}, \nu\right\rangle_{g_{0}}\leq|\nabla u|_{g}|\nu^{\top}|_{g_{0}},\end{equation}
where $\nabla$ is the gradient operator on $\mathcal{M}^{n}$.

\end{lem}

\begin{proof}By the Cauchy-Schwarz inequality, we have

\begin{equation}\label{yyuv}
\begin{aligned}\sum_{\alpha=1}^{n+p}\left\langle\nabla y^{\alpha}, \nabla u\right\rangle_{g}\left\langle\nabla y^{\alpha}, \nu\right\rangle_{g_{0}}&\leq\left(\sum_{\alpha=1}^{n+p}\left\langle\nabla y^{\alpha}, \nabla u\right\rangle_{g}^{2}\right)^{\frac{1}{2}}\left(\sum_{\alpha=1}^{n+p}\left\langle\nabla y^{\alpha}, \nu\right\rangle_{g_{0}}^{2}\right)^{\frac{1}{2}}.
\end{aligned}
\end{equation}
It follows from \eqref{uw-ine} that,

\begin{equation}\label{yu-ineq}\sum_{\alpha=1}^{n+p}\left\langle\nabla y^{\alpha}, \nabla u\right\rangle_{g}^{2}=|\nabla u|_{g}^{2}.\end{equation}
From \eqref{v-2}, \eqref{yyuv} and \eqref{yu-ineq}, we get \eqref{uv-ineq}. Therefore, we finish the proof of this lemma.
\end{proof}
From \eqref{n-ine}, we have
\begin{equation}
\int_{\Omega} u_{i}^{2} \sum_{\alpha=1}^{n+p}\left|\nabla y^{\alpha}\right|_{g}^{2} e^{\langle\nu,X\rangle_{g_{0}}}dv=n.
\end{equation}
According to \eqref{uw-ine}, one has

\begin{equation}\label{na-ui-2}
\sum_{\alpha=1}^{n+p}\left\langle\nabla y^{\alpha}, \nabla u_{i}\right\rangle_{g}^{2}=\left|\nabla u_{i}\right|_{g}^{2}.\end{equation}
It follows from \eqref{0-ine} that,
\begin{equation}\label{xx0-ineq}
\sum_{\alpha=1}^{n+p} \Delta y^{\alpha}\left\langle\nabla y^{\alpha}, \nabla u_{i}\right\rangle_{g} =\sum_{\alpha=1}^{n+p}\left\langle\Delta y^{\alpha}\nabla y^{\alpha}, \nabla u_{i}\right\rangle_{g}=0,
\end{equation}
and

\begin{equation}\label{xx-v}
\sum_{\alpha=1}^{n+p} \Delta y^{\alpha}\left\langle\nabla y^{\alpha}, \nu\right\rangle_{g_{0}}=\sum_{\alpha=1}^{n+p}\left\langle\Delta y^
{\alpha}\nabla y^{\alpha}, \nu\right\rangle_{g_{0}}=0.
\end{equation}
From \eqref{uv-ineq} and  \eqref{xx0-ineq}, we obtain

\begin{equation}\begin{aligned}\label{vxx-u}
\sum_{\alpha=1}^{n+p} \mathfrak{L}_{\nu} y^{\alpha}\left\langle\nabla y^
{\alpha}, \nabla u_{i}\right\rangle_{g}=\sum_{\alpha=1}^{n+p}\left(\Delta y^{\alpha}+\left\langle\nabla y^
{\alpha}, \nu\right\rangle_{g_{0}}\right)\left\langle\nabla y^{\alpha}, \nabla u_{i}\right\rangle_{g}\leq|\nabla u_{1}|_{g}|\nu^{\top}|_{g_{0}}.
\end{aligned}\end{equation}

Let $y^{1}, y^{2}, \ldots, y^{n+p}$ be the standard coordinate functions of $\mathbb{R}^{n+p}$ and
define an $((n+p) \times (n+p))$-matrix $D$ by
$D:=\left(d_{\alpha \beta}\right),$where
$d_{\alpha \beta}=\int_{\Omega} y^{\alpha} u_{1} u_{\beta+1} .$ Using the orthogonalization of Gram and Schmidt, we know that there exist an upper triangle matrix $R=\left(R_{\alpha \beta}\right)$ and an orthogonal matrix $Q=\left(\tau_{\alpha \beta}\right)$ such that $$R=QD,$$ i.e.,

\begin{equation*}\begin{aligned}
R_{\alpha \beta}=\sum_{\gamma=1}^{n+p} \tau_{\alpha \gamma} d_{\gamma \beta}=\int_{\Omega} \sum_{\gamma=1}^{n+p} \tau_{\alpha \gamma} y^{\gamma} u_{1} u_{\beta+1}=0,
\end{aligned}\end{equation*}
for $1 \leq \beta<\alpha \leq n+p$. Defining \begin{equation}\label{h-a}h_{\alpha}=\sum_{\gamma=1}^{n+p} \tau_{\alpha \gamma} y^{\gamma},\end{equation} we have $$\int_{\Omega} h_{\alpha} u_{1} u_{\beta+1}=0,$$ where $1 \leq \beta<\alpha \leq n+p .$
Since $Q$ is an orthogonal matrix, by lemma \ref{lem3.1} and lemma \ref{lem3.3}, we have the following lemma.

\begin{lem}\label{lem3.4}Under the above convention, we have
\begin{equation}\label{n-2}
\sum_{\alpha=1}^{n+p}\left|\nabla h
_{\alpha}\right|_{g}^{2}=n,\end{equation}

\begin{equation}\label{de-h}\sum_{\alpha=1}^{n+p}\left(\Delta h_{\alpha}\right)^{2}=n^{2}H^{2},\end{equation}

\begin{equation}\label{hhu} \sum_{\alpha=1}^{n+p} \Delta h_{\alpha}\left\langle\nabla h_{\alpha}, \nabla u_{1}\right\rangle_{g}=0,
\end{equation}

\begin{equation}\label{hhv}
\sum_{\alpha=1}^{n+p} \Delta h_{\alpha}\left\langle\nabla h_{\alpha}, \nu\right\rangle_{g_{0}}=0,\end{equation}

\begin{equation}\label{huhv} \sum_{\alpha=1}^{n+p}\left\langle\nabla h_{\alpha}, \nabla u_{1}\right\rangle_{g}\left\langle\nabla h_{\alpha}, \nu\right\rangle_{g_{0}}\leq|\nabla u_{1}|_{g}|\nu^{\top}|_{g_{0}}, \end{equation}

\begin{equation}\label{vv2} \sum_{\alpha=1}^{n+p}\left\langle\nabla h_{\alpha}, \nu\right\rangle_{g_{0}}^{2}=\left|\nu^{\top}\right|_{g_{0}}^{2}, \end{equation}
and

\begin{equation}\label{nab-u-2} \sum_{\alpha=1}^{n+p}\left\langle\nabla h_{\alpha}, \nabla u_{1}\right\rangle_{g}^{2}=\left|\nabla u_{1}\right|_{g}^{2}. \end{equation}

\end{lem}

\section{Proofs of Main Results}\label{sec4}

In this section, we would like to give the proofs of the main results.

From \eqref{de-h}, \eqref{hhv} and \eqref{vv2}, we obtain

\begin{equation}\label{nhv} \sum_{\alpha=1}^{n+p}\left( \mathfrak{L}_{\nu} h_{\alpha}\right)^{2}= n^{2}H^{2}+|\nu^{\top}|_{g_{0}}^{2}.\end{equation}
Utilizing \eqref{hhu} and \eqref{huhv}, one has

\begin{equation}\label{nab-u-nu}\sum_{\alpha=1}^{n+p} \mathfrak{L}_{\nu} h_{\alpha}\left\langle\nabla h_{\alpha}, \nabla u_{1}\right\rangle_{g}\leq|\nabla u_{1}|_{g}|\nu^{\top}|_{g_{0}}.
\end{equation}

\begin{lem} For any $i=1,2,\cdots k$ and $\alpha=1,2,\cdots,n+p$, let

\begin{equation}\label{widehat-upsilon}\widehat{\Upsilon}=\sum^{n+p}_{\alpha=1}\int_{\Omega}\Upsilon(h_{\alpha})e^{\langle\nu,X\rangle_{g_{0}}}dv,\end{equation}where function $\Upsilon$ is given by \eqref{Upsilon} and $h_{\alpha}$ is given by \eqref{h-a}. Then, we have

\begin{equation}\begin{aligned}\label{Upsilon-31}\widehat{\Upsilon}&\leq\int_{\Omega}\left[4\left|\nabla u_{1}\right|_{g}^{2}+u_{1}^{2}\left(n^{2}H^{2}+|\nu^{\top}|_{g_{0}}^{2}\right)\right]e^{\langle\nu,X\rangle_{g_{0}}}dv \\&\ \ \ \ \ \ + 4\left(\int_{\Omega}u_{1}^{2}|\nu^{\top}|_{g_{0}}^{2} e^{\langle\nu,X\rangle_{g_{0}}}dv\right)^{\frac{1}{2}}\left(\int_{\Omega}|\nabla u_{1}|_{g}^{2} e^{\langle\nu,X\rangle_{g_{0}}}dv\right)^{\frac{1}{2}},\end{aligned}\end{equation}
and

\begin{equation}\begin{aligned}\label{Upsilon-31-1}\widehat{\Upsilon}\leq\int_{\Omega}\left[6\left|\nabla u_{1}\right|_{g}^{2}+u_{1}^{2}\left(\left(n^{2}H^{2}+3|\nu^{\top}|_{g_{0}}^{2}\right)\right)\right]e^{\langle\nu,X\rangle_{g_{0}}}dv.\end{aligned}\end{equation}
\end{lem}

\begin{proof}  By \eqref{h-a} and \eqref{widehat-upsilon}, we have

\begin{equation}\begin{aligned}\widehat{\Upsilon}&=\sum^{n+p}_{\alpha=1}\int_{\Omega}\left(u_{1} \mathfrak{L}_{\nu} h_{\alpha}+2\langle\nabla h_{\alpha}, \nabla u_{1}\rangle_{g}\right)^{2}e^{\langle\nu,X\rangle_{g_{0}}}dv\\&=\int_{\Omega}\sum^{n+p}_{\alpha=1}\left(u^{2}_{1} (\mathfrak{L}_{\nu} h_{\alpha})^{2}+4u_{1}\mathfrak{L}_{\nu} h_{\alpha} \langle\nabla h_{\alpha},\nabla u_{1}\rangle_{g}+4 \langle\nabla h_{\alpha}, \nabla u_{1}\rangle_{g}^{2}\right)e^{\langle\nu,X\rangle_{g_{0}}}dv.\end{aligned}\end{equation}
From \eqref{hhu}, \eqref{nab-u-2}, \eqref{nhv} and \eqref{nab-u-nu}, we infer that

\begin{equation}\begin{aligned}\label{ineq-Up-4.17}\widehat{\Upsilon}\leq\int_{\Omega}\left[u^{2}_{1} (n^{2}H^{2}+|\nu^{\top}|_{g_{0}}^{2})+4|\nabla u_{1}|_{g}^{2}\right]e^{\langle\nu,X\rangle_{g_{0}}}dv+4\int_{\Omega}(u_{1}|\nu^{\top}|_{g_{0}})|\nabla u_{1}|_{g}e^{\langle\nu,X\rangle_{g_{0}}}dv.\end{aligned}\end{equation}
Furthermore, by Cauchy-Schwarz inequality, we have

\begin{equation}\begin{aligned}\label{cau}4\int_{\Omega}(u_{1}|\nu^{\top}&|_{g_{0}})|\nabla u_{1}|_{g}e^{\langle\nu,X\rangle_{g_{0}}}dv\\&\leq4\left(\int_{\Omega}(u_{1}|\nu^{\top}|_{g_{0}})^{2} e^{\langle\nu,X\rangle_{g_{0}}}dv\right)^{\frac{1}{2}}\left(\int_{\Omega}|\nabla u_{1}|_{g}^{2} e^{\langle\nu,X\rangle_{g_{0}}}dv\right)^{\frac{1}{2}}.\end{aligned}\end{equation}
 From \eqref{ineq-Up-4.17} and \eqref{cau}, we yield

\begin{equation*}\begin{aligned}\widehat{\Upsilon}&\leq\int_{\Omega}\left[4\left|\nabla u_{1}\right|_{g}^{2}+u_{1}^{2}\left(n^{2}H^{2}+|\nu^{\top}|_{g_{0}}^{2}\right)\right]e^{\langle\nu,X\rangle_{g_{0}}}dv \\&\ \ \ \ \ \ + 4\left(\int_{\Omega}(u_{1}|\nu^{\top}|_{g_{0}})^{2} e^{\langle\nu,X\rangle_{g_{0}}}dv\right)^{\frac{1}{2}}\left(\int_{\Omega}|\nabla u_{1}|_{g}^{2} e^{\langle\nu,X\rangle_{g_{0}}}dv\right)^{\frac{1}{2}}.\end{aligned}\end{equation*} By mean inequality, we obtain

\begin{equation}\label{mean}4\int_{\Omega}(u_{1}|\nu^{\top}|_{g_{0}})|\nabla u_{1}|_{g}e^{\langle\nu,X\rangle_{g_{0}}}dv\leq2\int_{\Omega}(u_{1}|\nu^{\top}|_{g_{0}})^{2} e^{\langle\nu,X\rangle_{g_{0}}}dv+2\int_{\Omega}|\nabla u_{1}|_{g}^{2} e^{\langle\nu,X\rangle_{g_{0}}}dv.\end{equation}
Therefore, by \eqref{ineq-Up-4.17} and \eqref{mean}, we derive that

\begin{equation*}\begin{aligned}\widehat{\Upsilon}&\leq\int_{\Omega}4\left|\nabla u_{1}\right|_{g}^{2}+u_{1}^{2}\left(n^{2}H^{2}+|\nu^{\top}|_{g_{0}}^{2}\right)e^{\langle\nu,X\rangle_{g_{0}}}dv \\&\ \ \ \ \ \ + 2\int_{\Omega}(u_{1}|\nu^{\top}|_{g_{0}})^{2} e^{\langle\nu,X\rangle_{g_{0}}}dv+2\int_{\Omega}|\nabla u_{1}|_{g}^{2} e^{\langle\nu,X\rangle_{g_{0}}}dv,\end{aligned}\end{equation*}which gives \eqref{Upsilon-31-1}. Therefore, we finish the proof of this lemma.

\end{proof}

By \eqref{n-2},  we have the following lemma.
\begin{lem} For any $i=1,2,\cdots k$ and $\alpha=1,2,\cdots,n+p$, let

\begin{equation*}\widehat{\Phi}=\sum^{n+p}_{\alpha=1}\int_{\Omega}\Phi(h_{\alpha})e^{\langle\nu,X\rangle_{g_{0}}
}dv,\end{equation*}where function $\Phi$ is given by \eqref{Phi}  and $h_{\alpha}$ is given by \eqref{h-a}. Then, we have

\begin{equation}\label{Phi-31}\widehat{\Phi}=n\int_{\Omega}u_{1} \mathfrak{L}_{\nu} u_{1}e^{\langle\nu,X\rangle_{g_{0}}}dv.\end{equation}

\end{lem}

Now, we give the proof of theorem \ref{thm1.1}.

\vskip 3mm
\noindent \emph{Proof of theorem} \ref{thm1.1}. From \eqref{lem-2.2}, noticing the definitions of $\widehat{\Phi}$ and $\widehat{\Upsilon}$,  we have

\begin{equation}\begin{aligned}\label{Sum-thm1.2}
\sum^{n+p}_{\alpha=1}\left(\Lambda_{\alpha+1}-\Lambda_{1}\right)^{\frac{1}{2}} \int_{\Omega}\left|u_{1} \nabla h_{\alpha}\right|_{g}^{2}e^{\langle\nu,X\rangle_{g_{0}}}dv
\leq\left(\frac{\delta}{2}+\frac{1}{2 \delta}\right)\widehat{\Upsilon} -\delta\widehat{\Phi}.
\end{aligned}\end{equation}By divergence theorem and Cauchy-Schwarz inequality, we conclude that

\begin{equation*}
 \int_{\Omega}\left|\nabla u_{1}\right|_{g}^{2}e^{\langle\nu,X\rangle_{g_{0}}}dv \leq \Lambda_{1}^{\frac{1}{2}},
\end{equation*}which gives
\begin{equation}\label{wide-Phi}
\widehat{\Phi}\geq- n \Lambda_{1}^{\frac{1}{2}}.
\end{equation}
Since eigenvalues are invariant under isometries, letting

\begin{equation*}
C_{1}=\frac{1}{4}\inf _{\sigma \in \Pi}\max_{\Omega}\left(n^{2}H^{2}\right),
\end{equation*}and

$$\widetilde{C}_{1}=\frac{1}{4}\max_{\Omega} |\nu^{\top}|_{g_{0}} ,$$
where $\Pi$ denotes the set of all isometric immersions from $\mathcal{M}^{n}$ into the  Euclidean space $\mathbb{R}^{n+p}$, by inequality \eqref{wide-Phi}, we infer that

\begin{equation}\begin{aligned}\label{wide-Upsilon}\widehat{\Upsilon}&\leq\int_{\Omega}\left[4\left|\nabla u_{1}\right|_{g}^{2}+u_{1}^{2}\left(n^{2}H
^{2}+|\nu^{\top}|_{g_{0}}^{2}\right)\right] e^{\langle\nu,X\rangle_{g_{0}}}dv\\&\quad+4\Lambda^{\frac{1}{4}}_{1}\left(\int_{\Omega}\left(u_{1}|\nu^{\top}|_{g_{0}}\right)^{2} e^{\langle\nu,X\rangle_{g_{0}}}dv\right)^{\frac{1}{2}}
\\&\leq 4 \left(\Lambda_{1}^{\frac{1}{2}}+4\widetilde{C}_{1}\Lambda_{1}^{\frac{1}{4}}+4\widetilde{C}_{1}^{2}+C_{1}\right).
\end{aligned}\end{equation}
For $\forall x \in \mathcal{M}^{n},$ by a transformation of orthonormal frame if necessary, it is not difficult to prove that, for any $\alpha$,

\begin{equation}\label{nab-h-alp}
\left|\nabla h_{\alpha}\right|_{g}^{2} \leq 1,
\end{equation} where $\alpha=1,2,\cdots,n+p$. It is clear that

\begin{equation}\begin{aligned}\label{Sum-3.21}
&\sum_{\alpha=1}^{n+p}\left(\Lambda_{\alpha+1}-\Lambda_{1}\right)^{\frac{1}{2}}\int_{\Omega}\left|u_{1}\nabla h_{\alpha}\right|_{g}^{2}e^{\langle\nu,X\rangle_{g_{0}}}dv\\ &
\geq \sum_{i=1}^{n}\left(\Lambda_{i+1}-\Lambda_{1}\right)^{\frac{1}{2}}\int_{\Omega}\left|u_{1}\nabla h_{i}\right|^{2}_{g}e^{\langle\nu,X\rangle_{g_{0}}}dv\\&\quad\quad+\left(\Lambda_{n+1}-\Lambda_{1}\right)^{\frac{1}{2}} \sum_{j=n+1}^{n+p}\int_{\Omega}\left|u_{1}\nabla h_{j}\right|_{g}^{2}e^{\langle\nu,X\rangle_{g_{0}}}dv.
\end{aligned}\end{equation}
Hence, from \eqref{n-2}, \eqref{nab-h-alp} and \eqref{Sum-3.21}, we infer that,

\begin{equation*}\begin{aligned}
&\sum_{\alpha=1}^{n+p}\left(\Lambda_{\alpha+1}-\Lambda_{1}\right)^{\frac{1}{2}}\int_{\Omega}\left|u_{1}\nabla h_{\alpha}\right|_{g}^{2}e^{\langle\nu,X\rangle_{g_{0}}}dv\\ &
\geq\sum_{i=1}^{n}\left(\Lambda_{i+1}-\Lambda_{1}\right)^{\frac{1}{2}}\int_{\Omega}\left|u_{1}\nabla h_{i}\right|^{2}_{g}e^{\langle\nu,X\rangle_{g_{0}}}dv\\&\quad\quad+\left(\Lambda_{n+1}-\Lambda_{1}\right)^{\frac{1}{2}}\left(n-\sum_{j=1}^{n}\int_{\Omega}\left|u_{1}\nabla h_{j}\right|_{g}^{2}\right)e^{\langle\nu,X\rangle_{g_{0}}}dv \\&
=\sum_{i=1}^{n}\left(\Lambda_{i+1}-\Lambda_{1}\right)^{\frac{1}{2}}\int_{\Omega}\left|u_{1}\nabla h_{i}\right|^{2}_{g}e^{\langle\nu,X\rangle_{g_{0}}}dv\\&\quad\quad+\left(\Lambda_{n+1}-\Lambda_{1}\right)^{\frac{1}{2}}\sum_{j=1}^{n}\left(1-\int_{\Omega}\left|u_{1}\nabla h_{j}\right|_{g}^{2}\right)e^{\langle\nu,X\rangle_{g_{0}}}dv \\&
\geq \sum_{i=1}^{n}\int_{\Omega}\left(\Lambda_{i+1}-\Lambda_{1}\right)^{\frac{1}{2}}\int_{\Omega}\left|u_{1}\nabla h_{i}\right|^{2}_{g}e^{\langle\nu,X\rangle_{g_{0}}}dv\\&\quad\quad+\sum_{j=1}^{n}\left(\Lambda_{j+1}-\Lambda_{1}\right)^{\frac{1}{2}}\int_{\Omega}\left(u_{1}^{2}-\left|u_{1}\nabla h_{j}\right|_{g}^{2}\right)e^{\langle\nu,X\rangle_{g_{0}}}dv,
\end{aligned}\end{equation*}which implies that

\begin{equation}\begin{aligned}\label{Sum-3.2}
\sum_{\alpha=1}^{n+p}\left(\Lambda_{\alpha+1}-\Lambda_{1}\right)^{\frac{1}{2}}\int_{\Omega}\left|u_{1}\nabla h_{\alpha}\right|_{g}^{2}e^{\langle\nu,X\rangle_{g_{0}}}dv
\geq\sum_{j=1}^{n}\left(\Lambda_{j+1}-\Lambda_{1}\right)^{\frac{1}{2}}.
\end{aligned}\end{equation}
Using \eqref{Sum-thm1.2}, \eqref{wide-Phi}, \eqref{wide-Upsilon} and \eqref{Sum-3.2},  we have

$$
\sum_{j=1}^{n}\left(\Lambda_{j+1}-\Lambda_{1}\right)^{\frac{1}{2}} \leq 4\left(\frac{\delta}{2}+\frac{1}{2 \delta}\right)\left( \Lambda_{1}^{\frac{1}{2}}+4\widetilde{C}_{1}\Lambda_{1}^{\frac{1}{4}}+4\widetilde{C}_{1}^{2}+C_{1}\right)+n \delta \Lambda_{1}^{\frac{1}{2}}.
$$
Taking
$$
\delta=\frac{\sqrt{ \Lambda_{1}^{\frac{1}{2}}+4\widetilde{C}_{1}\Lambda_{1}^{\frac{1}{4}}+4\widetilde{C}_{1}^{2}+C_{1}}}{\sqrt{\left(\frac{n}{2}+1\right) \Lambda_{1}^{\frac{1}{2}}+4\widetilde{C}_{1}\Lambda_{1}^{\frac{1}{4}}+4\widetilde{C}_{1}^{2}+C_{1}}},
$$
we get \eqref{z-ineq-1}. Therefore, it completes the proof of theorem \ref{thm1.1}.

$$\eqno\Box$$

According to theorem \ref{thm1.1}, we would like to give the proof of corollary \ref{corr1.1}.

\vskip 3mm
\noindent\emph{Proof of Corollary} \ref{corr1.1} Since

\begin{equation*}\begin{aligned}&4\left\{\left( \Lambda_{1}^{\frac{1}{2}}+4\widetilde{C}_{1}\Lambda_{1}^{\frac{1}{4}}+4\widetilde{C}_{1}^{2}+C_{1}\right)\left[\left(\frac{n}{2}+1\right) \Lambda_{1}^{\frac{1}{2}}+4\widetilde{C}_{1}\Lambda_{1}^{\frac{1}{4}}+4\widetilde{C}_{1}^{2}+C_{1}\right]\right\}^{\frac{1}{2}}\\&\leq n \Lambda_{1}^{\frac{1}{2}}+4\left( \Lambda_{1}^{\frac{1}{2}}+4\widetilde{C}_{1}\Lambda_{1}^{\frac{1}{4}}+4\widetilde{C}_{1}^{2}+C_{1}\right),\end{aligned}\end{equation*}where $C_{1}$ is given by

$$
C_{1}=\frac{1}{4}\inf _{\sigma \in \Pi}\max_{\Omega}\left(n^{2}H^{2}\right),
$$and $\widetilde{C}_{1}$ is given by $$\widetilde{C}_{1}=\frac{1}{4}\max_{\Omega} |\nu^{\top}|_{g_{0}} ,$$
from \eqref{z-ineq-1}, we then obtain
\begin{equation*}
\sum_{i=1}^{n}\left\{\left(\Lambda_{i+1}-\Lambda_{1}\right)^{\frac{1}{2}}-\Lambda_{1}^{\frac{1}{2}}\right\} \leq 4\left(\Lambda_{1}^{\frac{1}{2}}+4\widetilde{C}_{1}\Lambda_{1}^{\frac{1}{4}}+4\widetilde{C}_{1}^{2}+C_{1}\right).
\end{equation*}
This finishes the proof of corollary \ref{corr1.1}.

$$\eqno\Box$$

\vskip 3mm
\noindent \emph{Proof of corollary} \ref{corr1.2}. From \eqref{Upsilon-31-1}, we have

\begin{equation}\begin{aligned}\label{Upsilon-31-12}\widehat{\Upsilon}&\leq\int_{\Omega}\left[6\left|\nabla u_{1}\right|_{g}^{2}+u_{1}^{2}\left(n^{2}H^{2}+3|\nu^{\top}|_{g_{0}}^{2}\right)\right]e^{\langle\nu,X\rangle_{g_{0}}}dv.\end{aligned}\end{equation}
According to \eqref{lem-2.2} and the definitions of $\widehat{\Phi}$ and $\widehat{\Upsilon}$,  we derive that

\begin{equation}\begin{aligned}\label{Sum-thm1.2-1}
\sum^{n+p}_{\alpha=1}\left(\Lambda_{\alpha+1}-\Lambda_{1}\right)^{\frac{1}{2}} \int_{\Omega}\left|u_{1} \nabla h_{\alpha}\right|_{g}^{2}e^{\langle\nu,X\rangle_{g_{0}}}dv
\leq\left(\frac{\delta}{2}+\frac{1}{2 \delta}\right)\widehat{\Upsilon} -\delta\widehat{\Phi}.
\end{aligned}\end{equation}
Since eigenvalues are invariant under isometries, defining
\begin{equation} \label{C-0}
C_{2}= \frac{1}{6}\inf _{\sigma \in \Pi} \max _{\mathcal{M}^{n}}\left(n^{2}H^{2}+3|\nu^{\top}|_{g_{0}}^{2}\right),
\end{equation}
where $\Pi$ denotes the set of all isometric immersions from $\mathcal{M}^{n}$ into the Euclidean space $\mathbb{R}^{n+p}$, by divergence theorem and Cauchy-Schwarz inequality, we infer that

\begin{equation}\begin{aligned}\label{wide-Upsilon-1}\widehat{\Upsilon}\leq\int_{\Omega}\left[6\left|\nabla u_{1}\right|_{g}^{2}+u_{1}^{2}\left(n^{2}H
^{2}+3|\nu^{\top}|_{g_{0}}^{2}\right)\right]e^{\langle\nu,X\rangle_{g_{0}}}dv
\leq 6 \Lambda_{1}^{\frac{1}{2}}+6C_{2}.
\end{aligned}\end{equation}
Using \eqref{wide-Phi}, \eqref{Sum-3.2}, \eqref{Sum-thm1.2-1} and \eqref{wide-Upsilon-1}, we have

$$
\sum_{j=1}^{n}\left(\Lambda_{j+1}-\Lambda_{1}\right)^{\frac{1}{2}} \leq\left(\frac{\delta}{2}+\frac{1}{2 \delta}\right)\left(6 \Lambda_{1}^{\frac{1}{2}}+6C_{2}\right)+n \delta \Lambda_{1}^{\frac{1}{2}}.
$$
Taking
$$
\delta=\frac{\sqrt{6\Lambda_{1}^{\frac{1}{2}}+6C_{2}}}{\sqrt{(2 n+6) \Lambda_{1}^{\frac{1}{2}}+6C_{2}}},
$$
we have \eqref{z-ineq-1}. Therefore, it completes the proof of  corollary \ref{corr1.2}.

$$\eqno\Box$$

\vskip 3mm
\noindent \emph{Proof of corollary} \ref{corr1.3}. The method of the proof is the same as corollary \ref{corr1.1}. Hence, we omit it.

$$\eqno\Box$$

\section{Eigenvalue Inequalities on the Translating Solitons}\label{sec5}

\vskip5mm

In this section, we would like to discuss the eigenvalues of $\mathcal{L}_{II}^{2}$ on the complete translating solitons.

Firstly, let us consider a smooth family of immersions
$X_{t} = X(\cdot,t):\mathcal{M}^{n}\rightarrow \mathbb{R}^{n+p}$ with corresponding images $\mathcal{M}^{n}_{t} = X_{t}(\mathcal{M}^{n})$ such that the following mean curvature equation system \cite{H}:

\begin{equation}\label{MCF-Equa}{\begin{cases}
&\frac{d}{dt}X(x,t)=\textbf{H}(x,t), x\in \mathcal{M}^{n},\\
& X(\cdot,0) = X(\cdot),
\end{cases}}\end{equation}
is satisfied, where $\textbf{H}(x,t)$ is the mean curvature vector of $\mathcal{M}_{t}^{n}$ at $X(x, t)$ in $\mathbb{R}^{n+p}$.
 We  assume that $\nu_{0}$ is a constant vector with unit length and denote $\nu_{0}^{N}$ the normal projection
of $\nu_{0}$ to the normal bundle of $\mathcal{M}^{n}$ in $\mathbb{R}^{n+p}$. A submanifold $X:\mathcal{M}^{n}\rightarrow\mathbb{R}^{n+p}$ is said to be a translating soliton of the mean curvature flow \eqref{MCF-Equa}, if it
satisfies

\begin{equation}\label{tran} \textbf{H}=\nu_{0}^{N},\end{equation}which is a special solution of the mean curvature flow equations \eqref{MCF-Equa}.
Translating solitons also occur as Type-II singularity of the mean curvature flow equations \eqref{MCF-Equa}, which play an important role in the study
of the mean curvature flow \cite{AV}.
In \cite{Xin2}, Xin studied some basic properties of translating solitons: the volume growth, generalized maximum principle, Gauss maps and certain functions related to the Gauss maps. In addition, he carried out point-wise estimates and integral estimates for the squared norm of the second fundamental form. By utilizing these estimates,  Xin proved some rigidity theorems for translating solitons in the Euclidean space in higher codimension. Recently, Chen and Qiu \cite{ChQ} proved a nonexistence theorem for spacelike translating solitons. These results are established by using a new Omori-Yau maximal principle.

When $\nu_{0}$ is a unit vector field satisfying \eqref{tran},   $\mathfrak{L}_{\nu_{0}}$ exactly is an $\mathfrak{L}_{II}$   operator, which  is introduced by Xin
in  {\rm \cite{Xin2}} and similar to the $\mathfrak{L}$ operator introduced by  Colding and Minicozzi in {\rm \cite{CM}}. Therefore, $\mathfrak{L}_{\nu}$ operator can be viewed as a extension of  $\mathfrak{L}_{II}$ operator. As an application of theorem \ref{thm1.1}, we study the eigenvalues of bi-$\mathfrak{L}_{II}$ operator, which is denoted by $\mathfrak{L}_{II}^{2}$, on the complete translating solitons. In other words, we prove the following theorem.
\begin{thm}\label{thm4.1}
Let $\mathcal{M}^{n}$ be an $n$-dimensional complete translating soliton isometrically embedded into the Euclidean space $\mathbb{R}^{n+p}$ with mean curvature $H$. Then, eigenvalues of clamped plate problem \eqref{L-2-prob}
of the $\mathfrak{L}_{II}^{2}$ operator satisfy
\begin{equation}\begin{aligned}\label{z-ineq-2}
\sum_{i=1}^{n}\left(\Lambda_{i+1}-\Lambda_{1}\right)^{\frac{1}{2}}\leq4\left\{\left( \Lambda_{1}^{\frac{1}{2}}+ \Lambda^{\frac{1}{4}}_{1}+\frac{n^{2}}{4} \right)\left[\left(\frac{n}{2}+1\right) \Lambda_{1}^{\frac{1}{2}} + \Lambda^{\frac{1}{4}}_{1}+\frac{n^{2}}{4}\right]\right\}^{\frac{1}{2}}.
\end{aligned}\end{equation}

\end{thm}

\begin{proof}
Since $\mathcal{M}^{n}$ is an $n$-dimensional  complete  translator isometrically embedded into the Euclidean space $\mathbb{R}^{n+p}$, we have

\begin{equation}\label{3.6-1}\textbf{H}=\nu_{0}^{\perp}, \end{equation}  and

\begin{equation}\label{3.6-2}|\nu_{0}^{\top}|_{g_{0}}^{2}\leq|\nu_{0}|_{g_{0}}^{2}=1,\end{equation}  which implies that

\begin{equation}\label{3.7}n^{2}H^{2}+|\nu_{0}^{\top}|_{g_{0}}^{2}= n^{2}|\nu_{0}^{\perp}|_{g_{0}}^{2}+|\nu_{0}^{\top}|_{g_{0}}^{2}\leq n^{2}.\end{equation}
Uniting \eqref{3.6-1}, \eqref{3.6-2} and \eqref{3.7}, we yield \begin{equation}\label{3.8}\frac{1}{4}\int_{\Omega}u_{i}^{2}\left(n^{2}H^{2} +|\nu_{0}^{\top}|_{g_{0}}^{2}\right)e^{\langle\nu_{0},X\rangle_{g_{0}}}dv\leq \frac{n^{2}}{4}.\end{equation} Substituting \eqref{3.8} into \eqref{z-ineq-1}, we obtain

\begin{equation*}\begin{aligned}
\sum_{i=1}^{n}\left(\Lambda_{i+1}-\Lambda_{1}\right)^{\frac{1}{2}}\leq4\left\{\left( \Lambda_{1}^{\frac{1}{2}}+\frac{n^{2}}{4} + \Lambda^{\frac{1}{4}}_{1}\right)\left[\left(\frac{n}{2}+1\right) \Lambda_{1}^{\frac{1}{2}}+\frac{n^{2}}{4} + \Lambda^{\frac{1}{4}}_{1}\right]\right\}^{\frac{1}{2}}.
\end{aligned}\end{equation*}Therefore, we finish the proof of this theorem.

 \end{proof}

\begin{corr}\label{corr-last-1}
Under the same assumption as theorem {\rm \ref{thm4.1}},  eigenvalues of eigenvalue  problem \eqref{L-2-prob}
of $\mathfrak{L}_{II}^{2}$ operator satisfy
\begin{equation}\label{last-1}
\sum_{i=1}^{n}\left\{\left(\Lambda_{i+1}-\Lambda_{1}\right)^{\frac{1}{2}}-\Lambda_{1}^{\frac{1}{2}}\right\} \leq 4 \left(\Lambda_{1}^{\frac{1}{2}} + \Lambda^{\frac{1}{4}}_{1}+\frac{n^{2}}{4}\right).
\end{equation}

\end{corr}

\begin{proof}
 The method of proof is similar to corollary \ref{corr1.1}. Thus, we omit it.
\end{proof}

\begin{corr}\label{corr-last-2}
Under the same assumption as theorem {\rm \ref{thm4.1}},  for any $n\geq2$,  eigenvalues of clamped plate problem \eqref{L-2-prob}
of the $\mathfrak{L}_{II}^{2}$ operator satisfy

\begin{equation}\begin{aligned}\label{last-2}
\sum_{i=1}^{n}\left(\Lambda_{i+1}-\Lambda_{1}\right)^{\frac{1}{2}}\leq6\left\{\left( \Lambda_{1}^{\frac{1}{2}}+\frac{n^{2}}{6} \right)\left[\left(\frac{n}{3}+1\right) \Lambda_{1}^{\frac{1}{2}}+\frac{n^{2}}{6}\right]\right\}^{\frac{1}{2}}.
\end{aligned}\end{equation}
\end{corr}
\begin{proof}The method of the proof is similar to the proof of corollary \ref{corr1.2}. Thus, we omit it here.\end{proof}
According to corollary \ref{corr-last-3}, we can prove the following corollary.
\begin{corr}\label{corr-last-3}
Under the same assumption as theorem {\rm \ref{thm4.1}},  for any $n\geq2$,  eigenvalues of clamped plate problem \eqref{L-2-prob}
of the $\mathfrak{L}_{II}^{2}$ operator satisfy
\begin{equation}\label{last-3}
\sum_{i=1}^{n}\left\{\left(\Lambda_{i+1}-\Lambda_{1}\right)^{\frac{1}{2}}-\Lambda_{1}^{\frac{1}{2}}\right\} \leq 6\left(\Lambda_{1}^{\frac{1}{2}} + \frac{n^{2}}{6}\right).
\end{equation}

\end{corr}

\begin{rem} Since  inequality \eqref{z-ineq-2}, \eqref{last-1},  \eqref{last-2} and  \eqref{last-3} are
not dependent on the domain $\Omega$,  they are universal.\end{rem}

\section{Further Applications}\label{sec6}
In this section, we would like to give some further applications of theorem \ref{thm1.1}. Specially, we establish some eigenvalue inequalities on the minimal submanifolds of the
Euclidean spaces, unit spheres and projective spaces.

Firstly, we consider that $(\mathcal{M}^{n},g)$ is an $n$-dimensional complete minimal submanifold isometrically embedded into the $(n+p)$-dimensional Euclidean space $\mathbb{R}^{n+p}$. Then, we know that the mean curvature vanishes. Therefore,  one can deduce the following corollary from theorem \ref{thm1.1}.
\begin{corr}Let $(\mathcal{M}^{n},g)$ be an $n$-dimensional complete minimal submanifold isometrically embedded into the Euclidean space $\mathbb{R}^{n+p}$. Then,
eigenvalues of eigenvalue problem \eqref{L-2-prob} of $\mathfrak{L}_{\nu}^{2}$ operator satisfy
\begin{equation}\begin{aligned}\label{z-ineq-15}
\sum_{i=1}^{n}\left(\Lambda_{i+1}-\Lambda_{1}\right)^{\frac{1}{2}}\leq4\left\{\left( \Lambda_{1}^{\frac{1}{2}}+4\Lambda^{\frac{1}{4}}_{1}C_{3}+4C_{3}^{2}\right)\left[\left(\frac{n}{2}+1\right) \Lambda_{1}^{\frac{1}{2}}+4\Lambda^{\frac{1}{4}}_{1}C_{3}+4C_{3}^{2}\right]\right\}^{\frac{1}{2}},
\end{aligned}\end{equation}
where $C_{3}$ is given by
$$
C_{3}=\frac{1}{4}\max_{\Omega}|\nu^{\top}|_{g_{0}}.
$$  \end{corr}

Next, we consider that $(\mathcal{M}^{n},g)$ is an $n$-dimensional submanifold isometrically immersed in the unit sphere $\mathbb{S}^{n+p-1}(1) \subset \mathbb{R}^{n+p}$  with mean curvature vector $ \overline{\textbf{H}}$. We use $\overline{\Pi}$ to denote the set of all isometric immersions from $\mathcal{M}^{n}$ into  the unit sphere $\mathbb{S}^{n+p-1}(1)$. By theorem \ref{thm1.1}, we have the following corollary.
\begin{corr}\label{corr-6.2}
If $(\mathcal{M}^{n},g)$ be an $n$-dimensional submanifold isometrically immersed in the unit sphere $\mathbb{S}^{n+p-1}(1) \subset \mathbb{R}^{n+p}$  with mean curvature vector $\overline{\textbf{H}}$. Then,
eigenvalues of eigenvalue problem \eqref{L-2-prob} of $\mathfrak{L}_{\nu}^{2}$ operator satisfy
\begin{equation}\begin{aligned}\label{Sub-Sph}
\sum_{i=1}^{n}\left(\Lambda_{i+1}-\Lambda_{1}\right)^{\frac{1}{2}} \leq &4\left\{ \Lambda_{1}^{\frac{1}{2}}+ 4\widetilde{C} _{3} \Lambda_{1}^{\frac{1}{4}}+4\widetilde{C}_{4}^{2}+C_{4}\right\}^{\frac{1}{2}} \\
& \times\left\{\left( \frac{n}{2}+1\right) \Lambda_{1}^{\frac{1}{2}}+4\widetilde{C}_{4} \Lambda_{1}^{\frac{1}{4}}+4\widetilde{C}_{4}^{2}+C_{4}\right\}^{\frac{1}{2}},
\end{aligned}\end{equation}where

$$C_{4}=\frac{1}{4}\inf_{\overline{\sigma}\in\overline{\Pi}}\max_{\Omega}n^{2}(|\overline{\textbf{H}}|^{2}+1),$$
and

$$\widetilde{C}_{4}=\frac{1}{4}\max_{\Omega}|\nu^{\top}|_{g_{0}}.$$
\end{corr}

\begin{proof}
Since the unit sphere can be canonically imbedded into Euclidean space, we have
the following diagram:

\begin{equation*}\begin{aligned}
\xymatrix{
  \mathcal{M}^{n}\ar[dr]_{j\circ f} \ar[r]^{f}
                & \mathbb{S}^{n+p-1} \ar[d]^{j}  \\
                & \mathbb{R}^{n+p}              }\end{aligned} \end{equation*}
where $j: \mathbb{S}^{n+p-1}(1)\rightarrow \mathbb{R}^{n+p}$ is the canonical imbedding from the unit sphere $S^{n+p-1}(1)$ into $\mathbb{R}^{n+p},$ and  $f: \mathcal{M}^{n}\rightarrow \mathbb{S} ^{n+p-1}(1)$ is an isometrical immersion. Then, $j \circ f: \mathcal{M}^{n} \rightarrow \mathbb{R}^{n+p}$ is an isometric immersion from $\mathcal{M}^{n}$ to $\mathbb{R}^{n+p} .$ Let $\overline{\textbf{H}}$ and $\textbf{H}$ be the mean curvature vector fields of $f$ and $j \circ f,$ respectively; then
\[
\left| \textbf{H}\right|^{2}=|\overline{\textbf{H}}|^{2}+1.
\]
Applying  theorem \ref{thm1.1} directly, we can get \eqref{Sub-Sph}. Therefore, we finish the proof of corollary \ref{corr-6.2}.\end{proof}

In particular, we assume that $(\mathcal{M}^{n},g)$ is an $n$-dimensional unit sphere $\mathbb{S}^{n}(1)$, and then, the mean curvature equals to $1$. This is, $\left| \overline{\textbf{H}}\right|=0$, and thus we have $\left| \textbf{H}\right|=1$. Furthermore, by theorem \ref{thm1.1}, we obtain the following corollary.

\begin{corr}Let $(\mathcal{M}^{n},g)$ be an $n$-dimensional unit sphere $\mathbb{S}^{n}(1)$ and $\Omega$ is a bounded domain on  $\mathbb{S}^{n}(1)$. Then,
eigenvalues of eigenvalue problem \eqref{L-2-prob} of $\mathfrak{L}_{\nu}^{2}$ operator satisfy

\begin{equation*}\begin{aligned}
\sum_{i=1}^{n}\left(\Lambda_{i+1}-\Lambda_{1}\right)^{\frac{1}{2}}\leq4\left\{\left( \Lambda_{1}^{\frac{1}{2}}+4 C_{5}\Lambda^{\frac{1}{4}}_{1}+\frac{n^{2}}{4}+4C_{5}^{2}\right)\left[\left(\frac{n}{2}+1\right) \Lambda_{1}^{\frac{1}{2}}+4 C_{5}\Lambda^{\frac{1}{4}}_{1}+\frac{n^{2}}{4}+4C_{5}^{2}\right]\right\}^{\frac{1}{2}},
\end{aligned}\end{equation*}
where $C_{5}$ is given by

$$
C_{5}=\frac{1}{4}\max_{\Omega}|\nu^{\top}|_{g_{0}}.
$$
\end{corr}

Next, let us recall some results for submanifolds on the projective spaces. For more details, we refer the readers to \cite{Chb,CL}. Let $\mathbb{F}$ denote the field $\mathbb{R}$ of real numbers,
the field $\mathbb{C}$ of complex numbers or the field $\mathbb{Q}$ of quaternions. For convenience, we
introduce the integers

\begin{equation}\label{df}
d_{\mathbb{F}}=\operatorname{dim}_{\mathbb{R}}\mathbb{F}=\left\{\begin{array}{ll}
1, & \text { if } \mathbb{F} = \mathbb{R}; \\
2, & \text { if } \mathbb{F} = \mathbb{C}; \\
4, & \text { if } \mathbb{F} = \mathbb{Q}.
\end{array}\right.
\end{equation}Let us denote by $\mathbb{F}P^{m}$ the $m$-dimensional real projective space if $\mathbb{F}= \mathbb{R}$, the complex projective space with real dimension $2 m$ if $\mathbb{F}= \mathbb{C}$, and the quaternionic projective space with real dimension $4 m$ if $\mathbb{F}= \mathbb{Q}$, respectively. Here, the manifold $\mathbb{F}P^{m}$ carries a canonical metric so that the Hopf fibration $$\pi: \mathbb{S}^{d_{\mathbb{F}} \cdot(m+1)-1} \subset \mathbb{F}^{m+1} \rightarrow \mathbb{F}P^{m}$$ is a Riemannian submersion.
Hence, the sectional curvature of $\mathbb{R}P^{m}$ is $1$, the holomorphic sectional curvature is $4$ and the quaternion sectional  curvature is $4$. Let $$\mathcal{H}_{m+1}(\mathbb{F})=\left\{A \in \mathcal{A} _{m+1}(\mathbb{F}) \mid A^{*}:=\overline{^{t} A}=A\right\}$$ be the vector space of $(m+1) \times(m+1)$ Hermitian
matrices with coefficients in the field $\mathbb{F}$, where $\mathcal{A}$ denotes the space of all $(m+1)\times(m+1)$ matrices over $\mathbb{F}$. We can endow $\mathcal{H}_{m+1}( \mathbb{F})$ with the inner product
\[
\langle A, B\rangle=\frac{1}{2} \operatorname{tr}(A B),
\]
where tr $(\cdot)$ denotes the trace for the given $(m+1) \times(m+1)$ matrix. Clearly, the map $\psi: \mathbb{S} ^{d_{\mathbb{F}} \cdot(m+1)-1} \subset \mathbb{F} ^{m+1} \rightarrow$
$\mathcal{H}_{m+1}(\mathbb{F})$ given by
\[
\psi=\left(\begin{array}{llll}
\left|z_{0}\right|^{2} & z_{0} \overline{z_{1}} & \cdots & z_{0} \overline{z_{m}} \\
z_{1} \overline{z_{0}} & \left|z_{1}\right|^{2} & \cdots & z_{1} \overline{z_{m}} \\
\cdots & \cdots & \cdots & \cdots \\
z_{m} \overline{z_{0}} & z_{m} \overline{z_{1}} & \cdots & \left|z_{m}\right|^{2}
\end{array}\right)
\]
induces through the Hopf fibration an isometric embedding $\psi$ from $\mathbb{F}P^{m}$ into $\mathcal{H}_{m+1}( \mathbb{F}) .$ Moreover, $\psi\left( \mathbb{F}P^{m}\right)$ is a minimal submanifold of the hypersphere $\mathbb{S} \left(\frac{I}{m+1}, \sqrt{\frac{m}{2(m+1)}}\right)$ of $\mathcal{H}_{m+1}( \mathbb{F})$ with radius $\sqrt{\frac{m}{2(m+1)}}$ and
center $\frac{I}{m+1}$, where $I$ is the identity matrix.
In addition, we need a result as follows (cf. lemma 6.3 in Chapter 4 in \cite{Chb}):

\begin{lem} \label{lem-proj}Let $f: \mathcal{M}^{n}  \rightarrow \mathbb{F} P^{\text {m }}$ be an isometric immersion, and let $\widehat{\textbf{H}}$ and $\textbf{H}$ be the mean curvature
vector fields of the immersions $f$ and $\psi \circ f,$ respectively (here $\psi$ is the induced isometric embedding $\psi$ from $\mathbb{F}P^{m}$ into $\mathcal{H}_{m+1}( \mathbb{F})$ explained above). Then, we have
\[
\left| \textbf{H}\right|^{2}=|\widehat{\textbf{H}}|^{2}+\frac{4(n+2)}{3 n}+\frac{2}{3 n^{2}} \sum_{i \neq j} K\left(e_{i}, e_{j}\right),
\]
where $\left\{e_{i}\right\}_{i=1}^{n}$ is a local orthonormal basis of $\Gamma(T \mathcal{M}^{n})$ and $K$ is the sectional curvature of $\mathbb{F}P^{m}$ expressed $b y$
\[
K\left(e_{i}, e_{j}\right)=\left\{\begin{array}{ll}
1, & \text { if } \mathbb{F} = \mathbb{R}; \\
1+3\left(e_{i} \cdot J e_{j}\right)^{2}, & \text { if } \mathbb{F} = \mathbb{C}; \\
1+\sum_{r=1}^{3} 3\left(e_{i} \cdot J_{r} e_{j}\right)^{2}, & \text { if } \mathbb{F} = \mathbb{Q},
\end{array}\right.
\]
where $J$ is the complex structure of $\mathbb{C}P^{m}$ and $J_{r}$ is the quaternionic structure of $\mathbb{Q}P ^{m}$.\end{lem}

Therefore, one can infer from lemma \ref{lem-proj} that

\begin{equation}\label{H-3H}
\left|\textbf{H}\right|^{2}=\left\{\begin{array}{ll}
|\widehat{\textbf{H}}|^{2}+\frac{2(n+1)}{2 n}, & \text { for } \mathbb{R} P^{m}; \\
|\widehat{\textbf{H}}|^{2}+\frac{2(n+1)}{2 n}+\frac{2}{n^{2}} \sum_{i, j=1}^{n}\left(e_{i} \cdot J e_{j}\right)^{2} \leq|\widehat{\textbf{H}}|^{2}+\frac{2(n+2)}{n}, & \text { for } \mathbb{C} P^{m}; \\
|\widehat{\textbf{H}}|^{2}+\frac{2(n+1)}{2 n}+\frac{2}{n^{2}} \sum_{i, j=1}^{n} \sum_{r=1}^{3}\left(e_{i} \cdot J_{r} e_{j}\right)^{2} \leq|\widehat{\textbf{H}}|^{2}+\frac{2(n+4)}{n}, & \text { for } \mathbb{Q} P^{m}.
\end{array}\right.
\end{equation}Hence, it follows from \eqref{H-3H} that,

\begin{equation}\label{HH}
\left| \textbf{H}\right|^{2} \leq \widehat{H}^{2}+\frac{2\left(n+d_{\mathbb{F}}\right)}{n},
\end{equation}where $\widehat{H}$ denotes the mean curvature of $\mathcal{M}^{n}$ isometrically immersed into the projective space $\mathbb{F}P^{m}$, this is to say that,$$\widehat{H}=|\widehat{\textbf{H}}|.$$
We note that the equality in \eqref{HH} holds if and only if $\mathcal{M}^{n}$ is a complex submanifold of $\mathbb{C}P^{m}$ (for the case $\mathbb{C}P^{m}$ ) while $n \equiv 0(\bmod 4)$ and $\mathcal{M}^{n}$ is an invariant submanifold of $\mathbb{Q}P^{m}\left(\text { for the case } \mathbb{Q}P^{m}\right)$.  We use $\widehat{\Pi}$ to denote the set of all isometric immersions from $\mathcal{M}^{n}$ into a projective space $\mathbb{F}P^{m}$. Then, from theorem \ref{thm1.1}, we can prove the following corollary.
\begin{corr}\label{corr-6.4}
If $\mathcal{M}^{n}$ is isometrically immersed in a projective space $\mathbb{F}P^{m}$ with mean curvature vector $\widehat{\textbf{H}}$, Then,
eigenvalues of eigenvalue problem \eqref{L-2-prob} of $\mathfrak{L}_{\nu}^{2}$ operator satisfy

\begin{equation}\label{proj-inequa}
\begin{aligned}
&\sum_{i=1}^{n}\left(\Lambda_{i+1}-\Lambda_{1}\right)^{\frac{1}{2}} \\&\leq 4\left\{ \left(\Lambda_{1}^{\frac{1}{2}}+ 4\widetilde{C} _{5} \Lambda_{1}^{\frac{1}{4}}+4\widetilde{C}_{6}^{2}+C_{6}\right)
  \left[\left( \frac{n}{2}+1\right) \Lambda_{1}^{\frac{1}{2}}+4\widetilde{C}_{6} \Lambda_{1}^{\frac{1}{4}}+4\widetilde{C}_{6}^{2}+C_{6}\right]\right\}^{\frac{1}{2}},
\end{aligned}
\end{equation}
where $C_{6}$ is given by

$$
C_{6}=\frac{1}{4}\inf _{\widehat{\sigma} \in \widehat{\Pi}}\max_{\Omega}\left(n^{2}|\widehat{\textbf{H}}|^{2}+2n\left(n+d_{\mathbb{F} }\right)\right),
$$  and $\widetilde{C}_{6}$ is given by $$\widetilde{C}_{6}=\frac{1}{4}\max_{\Omega} |\nu^{\top}|_{g_{0}},$$ and
$d_{\mathbb{F}}=\operatorname{dim}_{ \mathbb{R} } \mathbb{F}$ defined by \eqref{df}.

\end{corr}

\begin{proof}
Since there is a canonical imbedding from $\mathbb{F} P^{m}( \mathbb{F} = \mathbb{R} , \mathbb{C} , \mathbb{Q} )$ to Euclidean space $\mathcal{H}_{m+1}( \mathbb{F} )$, then for compact manifold $\mathcal{M}^{n}$ isometrically immersed into the projective space $\mathbb{F} P^{m},$ we have the following diagram:

\begin{equation*}\begin{aligned}
\xymatrix{
  \mathcal{M}^{n}\ar[dr]_{\psi\circ f} \ar[r]^{f}
                & \mathbb{F}P^{m} \ar[d]^{\psi}  \\
                &\mathcal{H}_{m+1}(\mathbb{F})             }\end{aligned} \end{equation*}
where $\psi: \mathbb{F}P^{m} \rightarrow \mathcal{H}_{m+1}( \mathbb{F})$ denotes the canonical imbedding from $\mathbb{F}P^{m}$ into $\mathcal{H}_{m+1}( \mathbb{F}),$ and   $f: \mathcal{M}^{n} \rightarrow$
$\mathbb{F}P^{m}$ denotes  an isometric immersion from $\mathcal{M}^{n}$ to $\mathbb{F}P^{m}$. Then, $\psi \circ f: \mathcal{M}^{n} \rightarrow \mathcal{H}_{m+1}( \mathbb{F})$ is an isometric immersion from $\mathcal{M}^{n}$ to $\mathcal{H}_{m+1}( \mathbb{F})$. Applying \eqref{HH} and theorem \ref{thm1.1}, one can get \eqref{proj-inequa}. Thus, it completes the proof of corollary \ref{corr-6.4}.

\end{proof}


\begin{thebibliography}{99}

\bibitem{AB1}
M. S. Ashbaugh and R. D. Benguria, More bounds on eigenvalue ratios
for Dirichlet Laplacians in $n$ dimension. SIAM J. Math. Anal., 1993, {\bf  24}
(6): 1622-1651.

\bibitem{AB2}
M. S. Ashbaugh and R. D. Benguria, A sharp bound for the ratio of
the first two eigenvalues of Dirichlet Laplacians and extensions.
Ann. of Math., 1992,{\bf 135}(3): 601-628.

\bibitem{AB3}
M. S. Ashbaugh and R. D. Benguria, A second proof of the
Payne-P\'{o}lya-Weinberger conjecture. Comm. Math. Phys., 1992, {\bf  147}
(1): 181-190.

\bibitem{A1}
M. S. Ashbaugh, Isoperimetric and universal inequalities for eigenvalues. in Spectral theory and
geometry (Edinburgh,1998), E. B. Davies and Yu Safalov eds., London Math.Soc. Lecture Notes,
 {\bf 273} (1999), Cambridge Univ. Press, Cambridge: 95-139.

\bibitem{AV}
S. B. Angenent and J. J. L. Velazquez, Asymptotic shape of cusp singularities in curve shortening. Duke
Math. J., 1995, {\bf 77}(1): 71-110.


\bibitem{Bran}
J. J. A. M. Brands, Bounds for the ratios of the first three membrane eigenvalues. Arch. Rational Mech.
Anal., 1964, {\bf  16 }(4): 265-268.

\bibitem{Chb}
B.Y. Chen,  Total Mean Curvature and Submanifolds of Finite Type. World Scientific, Singapore (1984)

\bibitem{CC}
D. Chen and Q.-M. Cheng, Extrinsic estimates for eigenvalues of the
Laplace operator. J. Math. Soc. Japan, 2008, {\bf 60}(2): 325-339.

\bibitem{CL}
D. Chen and H.  Li,  The sharp estimates for the first eigenvalue of Paneitz operator in 4-manifold. arXiv preprint,
arXiv:1010.3102 (2010)

\bibitem{CZ}
D. Chen and T. Zheng, Bounds for ratios of the membrane eigenvalues. J. Diff. Eqns., 2011, {\bf 250}(3): 1575-1590.

\bibitem{ChQ}
Q. Chen and H. Qiu, Rigidity of self-shrinkers and translating solitons of mean curvature flows. Adv. Math., 2016, {\bf 294}: 517-531.

\bibitem {CQian}
Z.-C. Chen and C.-L. Qian, Estimates for discrete spectrum of Laplacian operator with any
order, J. China Univ. Sci. Tech. 20 (1990), 259-266.

\bibitem{CHW}
Q.-M. Cheng, G. Y. Huang and G. X. Wei, Estimates for lower order
eigenvalues of a clamped plate problem. Calc. Var. Part. Diff.
Equa., {\bf 38} (2010), 409-416.

\bibitem {CQ}
Q.-M. Cheng and X. Qi, Eigenvalues of the Laplacian on Riemannian manifolds. Inter. J. Math., 2012, {\bf  23}(07): 1250067.


\bibitem {CY3}
Q.-M. Cheng, H.C. Yang, Inequalities for eigenvalues of a clamped plate problem, Trans. Amer. Math. Soc., 2006, {\bf  358}
: 2625-2635.

\bibitem{Chit}
J. Clutterbuck, O. Schn\"{u}rer and F. Schulze, Stability of translating solutions to mean curvature flow. Calc.
Var. Par. Diff. Equs.,  2007, {\bf29}(3): 281-293.

\bibitem{CM}
T. H. Colding and W. P. Minicozzi II, Generic mean curvature flow I;
Generic Singularities. Ann. of Math., 2012, {\bf 175} (2): 755-833.

\bibitem{HP}
F. N. Hile and  M. H. Protter, Inequalities for eigenvalues of the
Laplacian. Indiana Univ. Math. J., 1980, {\bf  29}(4): 523-538.

\bibitem{HY}
G. N. Hile and R. Z. Yeh, Inequalities for eigenvalues of the biharmonic operator, Pacific J.
Math., 1984, {\bf  112}: 115-133.

\bibitem{Hook}
S. M. Hook, Domain independent upper bounds for eigenvalues of elliptic operator, Trans.
Amer. Math. Soc., 1990, {\bf 318}: 615-642.

\bibitem{H}
G. Huisken, Asymptotic behavior for singularities of the mean
curvature flow. J. Diff. Geom., 1990, {\bf  31}(1): 285-299.

\bibitem{LP}
M. Levitin and L. Parnovski, Commutators, spectral trace identities,
and universal estimates for eigenvalues, J. Funct. Anal., 2002, \textbf{192}: 425-445.

\bibitem{Mar}
P. Marcellini, Bounds for the third membrane eigenvalue. J. Diff. Eqns., 1980, {\bf 37}(3): 438-443.


\bibitem{PPW2}
L. E.Payne, G. P\'{o}lya  and H. F. Weinberger, On the ratio of
consecutive eigenvalues. J. Math. and Phys., 1956, {\bf  35}(1-4): 289-298.


\bibitem{Sun}
H. J. Sun, Yang-type inequalities for weighted eigenvalues of a
second order uniformly elliptic operator with a nonnegative
potential. Proc. Amer. Math. Soc., 2010, {\bf  138} (8): 2827-2838.

\bibitem{SCY}
H.-J. Sun, Q.-M. Cheng and H.-C. Yang, Lower order eigenvalues of Dirichlet Laplacian. Manuscripta Math., 2008
{\bf  125}(2):  139-156.


\bibitem{WX1}
Q. Wang and C. Xia,  Universal bounds for eigenvalues of the biharmonic operator on Riemannian manifolds.
J. Funct. Anal., 2007, {\bf 245}: 334-352.

\bibitem{WX2}
Q. Wang and C. Xia, Universal bounds for eigenvalues of the biharmonic operator. J. Math. Anal. Appl., 2010,
{\bf 364}: 1-17.

\bibitem{WX3}
Q. Wang and C. Xia, Inequalities for eigenvalues of a clamped plate problem. Calc. Var. Partial Differ. Equ., 2011,
{\bf 40}: 273¨C289.

\bibitem{Xin2}
Y. L. Xin, Translating soliton of the mean curvature flow. Calc.
Var. Par. Diff. Equs., 2015,  {\bf  54}(2):1995-2016.


\end{thebibliography}
\end{document}